\documentclass[11pt]{article}
\usepackage[utf8]{inputenc}
\usepackage[T1]{fontenc}
\usepackage{latexsym,amssymb,amsmath,amsfonts,amsthm}
\usepackage{graphics}
\usepackage[section]{placeins}\usepackage{graphicx}
\usepackage{mathrsfs}
\usepackage{subfigure}
\usepackage{float}
\usepackage{color}
\usepackage{epstopdf}
\newcommand{\R}{{\mathbb R}}

\newcommand{\N}{{\mathbb N}}
\newcommand{\C}{{\mathbb C}}

\newcommand{\be}{\begin{eqnarray}}
\newcommand{\ben}{\begin{eqnarray*}}
\newcommand{\en}{\end{eqnarray}}
\newcommand{\enn}{\end{eqnarray*}}

\newcommand{\pa}{\partial}

\newcommand{\real}{{\rm Re\,}}
\newcommand{\ima}{{\rm Im\,}}

\newcommand{\s}{\mathbb{S}}
\newcommand{\G}{\Gamma}

\newcommand{\range}{{\rm Range}}

\newtheorem{theorem}{Theorem}[section]
\newtheorem{thm}{Theorem}[section]
\newtheorem{cor}{Corollary}[section]
\newtheorem{lem}{Lemma}[section]

\newtheorem{rem}{Remark}[section]

\definecolor{rot}{rgb}{0,0,0}
\definecolor{hw}{rgb}{0,0,0}
\definecolor{rot1}{rgb}{0,0,0}
\newcommand{\tcr}{\textcolor{rot}}
\newcommand{\rot}{\textcolor{rot1}}
\newcommand{\tcb}{\textcolor{hw}}

\begin{document}
\renewcommand{\theequation}{\arabic{section}.\arabic{equation}}
\begin{titlepage}
\title{\bf Inverse wave-number-dependent source problems for the Helmholtz equation}

\author{%Roland Griesmarier\thanks{Institute for Applied and Numerical Mathematics, Karlsruhe Institute of Technology, Germany. ({\tt roland.griesmaier@kit.edu})} \and
Hongxia Guo\thanks{School of Mathematical Sciences and LEMP, Nankai University, 300071 Tianjin, China. ({\tt hxguo@nankai.edu.cn})}  \and Guanghui Hu\thanks{Corresponding author: School of Mathematical Sciences and LEMP, Nankai University, 300071 Tianjin, China. ({\tt ghhu@nankai.edu.cn})}
}
\date{}
\end{titlepage}
\maketitle

%\vspace{.2in}

\begin{abstract}
This paper is concerned with the multi-frequency factorization method for imaging the support of a wave-number-dependent source function. It is supposed that the source function is given by the \textcolor{rot}{inverse} Fourier transform of some time-dependent source with a priori given radiating period. Using the multi-frequency far-field data at a fixed observation direction, we provide a computational criterion for characterizing the smallest strip containing the support and perpendicular to the observation direction. The far-field data from sparse observation directions can be used to recover a $\Theta$-convex polygon of the support. %Uniqueness in recovering the convex hull of the support is obtained as a by-product of the reconstruction method with the  data of all observation directions.
The inversion algorithm is proven valid even with multi-frequency near-field data in three dimensions. The connections to time-dependent inverse source problems are discussed in the near-field case. %We also comment on possible extensions to source functions with two disconnected supports.
Numerical tests in both two and three dimensions are implemented to  show effectiveness and feasibility of the approach. This paper provides numerical analysis for a frequency-domain approach to recover the support of an admissible class of time-dependent sources.

\vspace{.2in} {\bf Keywords: Inverse source problem, Helmholtz equation, wave-number-dependent sources, multi-frequency data, factorization method.
}
\end{abstract}

\section{Introduction and problem formulation}
Consider the time-dependent acoustic wave radiating from a source term in an isotropic and  homogeneous medium
\be\label{wave}\begin{array}{lll}
\partial_t^2U(x,t)=\Delta U(x,t)\tcr{+}S(x,t),&&\quad (x,t)\in \R^3\times \R_+,\\
U(x,0)=\partial_tU(x,0)=0,&&\quad x\in \R^3,
\end{array}
\en
where  $\mbox{supp}\, S(x,t)=\textcolor{rot}{\overline D}\times \rot{[}t_{\min}, t_{\max}\rot{]}\subset \R^3\times\R_+$ with $t_{\max}>t_{\min}\geq 0$. The wave speed in the background medium has been normalized to be one. We suppose that $D\subset \R^3$ is a bounded Lipschitz domain such that $\R^3\backslash\overline{D}$ is connected and that $S(x,t)\in C([t_{\min}, t_{\max}], L^\infty(D))$ is a \tcb{real-valued function} fulfilling the \textcolor{rot}{positivity} constraint
\be\label{F}
S(x, t)\geq c_0>0\qquad\mbox{a.e.}\; x\in \overline{D},\quad t\in \textcolor{rot}{[}t_{\min}, t_{\max}\textcolor{rot}{]}.
\en
The above condition \eqref{F} implies that  \textcolor{rot}{the location and shape of the time-dependent source support $D$} does not vary along with the time variable.
 The time interval $[t_{\min}, t_{\max}]\subset\R_+$ represents the duration period for source radiating.
The solution $U$ can be given explicitly as the convolution of the fundamental solution to the wave equation with the source term, that is,

\be\label{U}
U(x,t)=G(x, t)*S(x,t):=\int_{\R_+}\int_{\R^3} G(x-y; t-\tau) S(y,\tau) dy d\tau
%&=&\int_{\R_+}  \int_D\frac{\delta(t-s-|x-y|)}{4\pi |x-y|} S(y,\tau) dy d\tau.
\en
where
%\ben
$G(x, t)=\frac{\delta(t-|x|)}{4\pi |x|}$ is the Green's function to the wave equation in 3D.
%,\quad x\in \R^3\backslash\{0\},\;t\in \R_+,\; t\neq |x|.
%\enn
Taking the \textcolor{rot}{inverse} Fourier transform of $U(x, t)$ with respect to the time variable, one deduces from \eqref{U} that
\be\label{expression-w}
w(x,k):=\frac{1}{\sqrt{2\pi}}\int_{\R} U(x,t)e^{\textcolor{rot}{{\rm i} k t}}dt=\int_{\R^3}\Phi_k(x, y) f(y,k)\,dy,
%w(x,k):=\frac{1}{\sqrt{2\pi}}\int_{\R} U(x,t)e^{-{\rm i} k t}dt=\frac{1}{\sqrt{2\pi}}\int_{\R_+} U(x,t)e^{-{\rm i} k t}dt=\int_{\R^3}\Phi_k(x, y) f(y,k)\,dy,
\en
where $\Phi_{k}(x,y)=\frac{e^{{\rm i}k |x-y|}}{4\pi|x-y|}$ and $f(y, k)$ denote respectively the \textcolor{rot}{inverse} Fourier transforms of the fundamental solution $G(x-y; t)$ and $S(y, t)$. %By the wave equation of the fundamental solution $G(x, t)$,
%\ben
%\partial_t^2G(x-y,t)-\Delta G(x-y,t)=\delta(x-y)\delta(t),&&\quad x\in \R^3\backslash\{y\},\quad t>0, \\
%G(x-y,0)=\partial_tG(x-y,0)=0,&&\quad x\in \R^3\backslash\{y\},
%\enn
%we conclude that $\Phi_{k}(x, y)$
%coincides with the fundamental solution to the Helmholtz equation $(\Delta+k^2)w=0$, given by
%\ben
%\Phi_{k}(x,y)=\frac{e^{{\rm i}k |x-y|}}{4\pi |x-y|},\quad x\neq y,\quad x,y\in \R^3.
%\enn
It is well known that $\Phi_{k}(x,y)$ satisfies the Sommerfeld radiation condition. By the assumption of $S$, we have
\be\label{fxt}
f(x,k)=\frac{1}{\sqrt{2\pi}}\int_{\R} S(x,t)e^{\textcolor{rot}{{\rm i} k t}}dt=\frac{1}{\sqrt{2\pi}}\int_{t_{\min}}^{t_{\max}} S(x,t)e^{\textcolor{rot}{{\rm i} k t}}dt,
\en
which is compactly supported on $\overline{D}$ with respect to the space variables. Moreover, $f(\cdot, k)\in L^2(D)$ for any $k>0$. Therefore, it follows from \eqref{expression-w} that $w(\cdot, k)\in H^2_{loc}(\R^3)$ satisfies
\be\label{Helmholtz}
&\Delta w(x, k)+k^2 w(x, k)=-f(x, k), \quad  x\in \R^3,\; k>0,\\ \label{SRC}
&\lim_{r\rightarrow\infty}r (\partial_r w-{\rm i} k  w)=0,\quad  r=|x|,
\en
where the limit \eqref{SRC} holds uniformly in all directions $\hat{x}=x/|x|\in \s^2:=\{x\in \R^3: |x|=1\}$.
The Sommerfeld radiation condition \eqref{SRC}  gives rise to the following asymptotic behavior at infinity:
\be\label{far-field}
w(x)=\frac{e^{{\rm i} k |x|}}{4\pi|x|}\left\{w^\infty(\hat{x},k)+O\left(\frac{1}{|x|}\right)\right\}\quad\mbox{as}\quad|x|\rightarrow\infty,
\en
where $w^\infty(\cdot, k)\in C^\infty(\s^2)$ is referred to as the
 far-field pattern (or scattering amplitude) of $w$. It is well known that the function  $\hat{x}\mapsto w^\infty(\hat{x}, k)$ is real analytic on $\s^2$, where $\hat{x}\in \s^2$ is usually called the observation direction.

 By \eqref{expression-w}, the far-field pattern $w^\infty$ of $w$ can be expressed as
\be\label{u-infty}
w^\infty(\hat{x}, k)=\int_{D} e^{-{\rm i}k\hat{x}\cdot y} f(y, k)\,dy,\quad \hat{x}\in \s^2,\quad k>0.
\en
\textcolor{rot}{Furthermore, substituting (\ref{fxt}) into (\ref{u-infty}),  we obtain the expression of the far-field pattern $w^{\infty}$ in the frequency domain as follow
\be
w^\infty(\hat{x}, k)=\frac{1}{\sqrt{2\pi}}\int_{t_{\min}}^{t_{\max}}\int_{D} e^{-{\rm i}k(\hat{x}\cdot y-t)} S(x,t)\,dy\,dt,\quad \hat{x}\in \s^2,\quad k>0.
\en}

Noting that the time-dependent source $S$ is real valued, we have $f(x, -k)=\overline{f(x,k)}$  and thus $w^\infty(\hat{x}, -k)=\overline{w^\infty(\hat{x}, k)}$ for all $k>0$.
Let $0\leq k_{\min}<k_{\max}$ and denote by $(k_{\min}, k_{\max})$ the bandwidth of wave-numbers of the Helmholtz equation.
In this paper we are interested in the following inverse problem: determine the position and shape of the support $D$ from knowledge of the multi-frequency far-field patterns
$\left\{w^\infty(\hat{x}_j, k): k\in(k_{\min}, k_{\max}), \, j=1,2,\cdots, J \right\}.$
%We shall establish a factorization method for imaging $D$ from sparse far-field measurements  at multiple frequencies.

If the source term is independent of frequencies (which corresponds to the critical case that $S(x,t)=s(x)\delta(t)$ and $f(x, k)=s(x)),$%; see Remark \ref{rem3.2}),
the far-field pattern given by \eqref{u-infty} is nothing  else but the Fourier transform of the space-dependent source term $f$ at the Fourier variable  $\xi=k\hat{x}\in \R^3$ multiplied by some constant.
Since $f$ is compactly supported in $D$,  its Fourier transform is analytic in $\xi\in \R^3$. Hence, the far-field measurements over an interval of frequencies and observation directions uniquely determine the source function and also its support. A wide range of literature is devoted to inverse wavenumber-independent source problems with multi-frequency data, for example,  uniqueness proofs and increasing stability analysis with near-field measurements \cite{BLLT, BLT10, CIL, EV09, LY} and a couple of numerical schemes such as iterative method, Fourier method and test-function method
for recovering the source function \cite{BLLT, BLRX,  EV09, ZG}  and sampling-type methods for imaging the support  \cite{AHLS, GS, JLZ2019}.
On the other hand,  the inverse source problem with the measurement data at a single frequency becomes severely ill-posed.
It is impossible in general to determine a source function (even its support) from a single far-field pattern due to the existence of non-radiating sources; see e.g., \cite{BC, DW1973, DS1982} for non-uniqueness examples. In a series papers by Kusiak and Sylvester \cite{KS03, Sy06,  SK05}, the concept of convex scattering support has been introduced to define the smallest convex set that carries a single far-field pattern.
It was shown in \cite{B, HL2020} that a convex-polygonal source support and an admissible class of analytic source functions can be uniquely determined by a single far-field pattern. Numerical schemes such as the enclosure method \cite{Ik99} and one-wave factorization method \cite{GH22} were proposed for imaging the support of such convex-polygonal sources. The filtered backprojection method \cite{GHT12, GHT13} and a hybrid method involving iterative and range test \cite{AKS} were also investigated with a single far-field measurement.

In contrast to vast \rot{literatures} for space-dependent source terms, little is known if the source function depends on both frequency/wave-number and spatial variables. Here we assume that the dependence on the frequency is unknown. %In one of the author's previous work \cite{AHLS}, a direct imaging method was examined for recovering the support of a wave-number-dependent source, but numerical solutions are not satisfactory in comparision with the results for wavenumber-independent sources.
One can see the essential difficulties from the expression \eqref{u-infty}, where the far-field pattern is no longer the Fourier transform of \rot{the source function}. Hence, most existing methods cannot be straightforwardly carried over to \rot{frequency-number-dependent source terms}.
In this paper, we consider an inverse frequency-dependent source problem originating from inverse time-dependent source problems. \rot{The temporal function is supposed to be unknown, but the starting and terminal time points for radiating are given.}
Consequently, the source term \rot{takes} a special integral form of the time-dependent source function (see \eqref{fxt}) with  {\rm a priori} given source radiating period $\textcolor{rot}{[}t_{\min}, t_{\max}\textcolor{rot}{]}$.  This is motivated by the Fourier method of \cite{HKLZ19, HK20, HKZ2020} for proving uniqueness in determining \rot{the source function of} inhomogeneous hyperbolic equations with vanishing initial data. In these works the inverse time-dependent problems were reduced to equivalent problems in the time-harmonic regime with multi-frequency data.
The proposed factorization scheme seems not applicable to general wave-number-dependent sources, because we do not know how to get a desirable factorization form of the far-field operator. Confined by such source functions, we \rot{think} it is non-trivial to extend our method to inverse medium scattering problems with multi-frequency data. We refer to \cite{FSY, G11, GH15, GS17, JLZ19, JLZ, P10} for the application of the sampling-type methods to nonlinear inverse problems modeled by the Helmholtz equation.

The multi-static factorization method \cite{K98, KG08}, which was proposed by A. Kirsch in 1998 , %has been extensively studied in various inverse time-harmonic scattering problems using far-field patterns over all observation directions at a fixed frequency. It
provides a necessary and sufficient criterion for precisely characterizing
the shape and location of  a scattering obstacle, utilizing the multi-static spectral system of the
far-field operator. The multi-frequency factorization method was rigorously justified in \cite{GS} for recovering the smallest strip $K_D^{(\hat{x})}$ that contains the support $D$ of a wave-number-independent source and is orthogonal to the observation direction $\hat{x}$. Moreover, the $\Theta$-convex polygon of the support can be recovered from the multi-frequency far-field data over sparse observation directions.
The aim of this paper is to establish the analogue of the multi-frequency factorization method \cite{GS} for imaging the support of a wave-number-dependent source function of the form \eqref{fxt}.  We prove a new range identity for connecting ranges of the far-field operator $F$ and the 'data-to-pattern' operator $L$. This yields a computational criterion for characterizing the $\Theta$-convex hull of $D$ using the multi-frequency far-field data over sparse observation directions; see Theorem \ref{TH-hull}. %In this sense it inherits merits of the multi-static factorization method.

If the near-field measurement data are available in three dimensions, the reconstruction scheme can be used for \rot{recovering} the minimum and maximum distance between the support and a measurement position. The connection between
the near-field factorization method  and the time-dependent wave radiating problems will be discussed in Section \ref{sec4}.
In two dimensions, the factorization method with far-field data still remains valid, but the near-field version does not hold true any more, perhaps due to the lack of Huygens principle.
It is worthy noting that the wave-number-dependence of sources makes this paper quite different from \cite{GS}.  It is necessary to know the radiating period $\textcolor{rot}{[}t_{\min}, t_{\max}\textcolor{rot}{]}$ of the time-dependent source in advance. %If such an information is not available, we do not know how to establish the factorization method in the frequency domain.
Physically, this can be explained by the reason that arrival and terminal time points of wave signals at an observation point are available; we refer to Section \ref{sec4} for the physical interpretation. However, the a priori information on the radiating period $\rot{[}t_{\min}, t_{\max}\rot{]}$ can be relaxed to the condition that either $t_{\min}$ or $t_{\max}$ is known and this will be studied in our future works.
The reconstruction method considered here can be regarded as a frequency-domain method for recovering the support of a time-dependent source fulfilling the \rot{positivity} condition \eqref{F}. \rot{The novelty of this paper lies in the establishment of the multi-frequency factorization method for an inverse wave-number-dependent source problem. This has been achieved by considering a special kind of source terms. The Factorization scheme for general source terms, including inverse medium scattering problems, still remains open.}

The remaining part is organized as follows. In Section \ref{sec2}, the concept of multi-frequency far-field operator is introduced and a new range identity is verified. Section \ref{sec3} is devoted to the choice of test functions for characterizing the strip $K_D^{(\hat{x})}$ through the 'data-to-pattern' operator $L$. In Section \ref{sec4} we define indicator functions using the far-field and near-field data measured at one or several observation directions. We also present numerical tests for imaging
two disconnected supports in Section 5 for further explanations. %in Section \ref{num}. The readers are referred to the arXiv version of this paper for explanations.
%We comment on possible extensions of our reconstruction method  to source functions with in Section \ref{mul-compo} and numerical tests will implemented .

Below we introduce some notations to be used throughout this paper. Unless otherwise stated, we always suppose that $D$ is connected and bounded.
Given $\hat{x}\in \s^2$, we define
%\ben
$\hat{x}\cdot D:=\{t\in \R: t=\hat{x}\cdot y\;\mbox{for some}\; y\in D\}\subset \R.$
%\enn
Hence, $(\inf(\hat{x}\cdot D),  \sup(\hat{x}\cdot D))$ must be a finite and connected interval on the real axis.
A ball centered at $y\in \R^3$ with the radius $\epsilon>0$ will be denoted as $B_\epsilon(y)$. For brevity we write $B_\epsilon=B_\epsilon(0)$ when the ball is centered at the origin.
Obviously, $\hat{x}\cdot B_\epsilon(y)=(\hat{x}\cdot y-\epsilon,\hat{x}\cdot y+\epsilon)$.
In this paper the one-dimensional Fourier and inverse Fourier transforms are defined respectively by
\ben
(\mathcal{F}f)(k)=\frac{1}{\sqrt{2\pi}}\int_{\R}f(t)e^{-{\rm i} k t}\,dt,\quad
(\mathcal{F}^{-1}v)(t)=\frac{1}{\sqrt{2\pi}}\int_{\R}v(k)e^{{\rm i} k t}\,d\rot{k}.
\enn

\section{Factorization of far-field operator and a new range identity}\label{sec2}
Following the ideas of \cite{GS}, we introduce the central frequency  $k_c$ and  half of the bandwidth of the given data as
%\ben
$k_c:=(k_{\min}+k_{\max})/2$, $K :=(k_{\max}-k_{\min})/2.$
%\enn
For every
fixed  $\hat{x}\in \s^2$, we define the far-field operator by
\be \label{def:F}
(F\phi)(\tau)=(F^{(\hat{x})}\phi)(\tau)&:=&\int_{0}^{K } w^\infty(\hat{x}, k_c+\tau-s)\,\phi(s)\,ds, \quad \rot{\tau \in (0,K)}.%\\ \nonumber
%&=&\int_{0}^{K }\phi(s)\int_D e^{-{\rm i} (k_c+\tau-s)\hat{x}\cdot y}f(y,k_c+\tau-s)\,dy\,ds.
\en
Since $w^\infty(\hat{x},k)$ is analytic with respect to the wave number $k\in \R$, the operator
$F^{(\hat{x})}: L^2(0, K )\rightarrow L^2(0, K )$ is bounded.
For notational convenience we introduce the space
$X_D:=L^2(D\times(t_{\min}, t_{\max})).$
Denote by $\langle \cdot,  \cdot\rangle_{X_D}$ the inner product over $X_D$.
Below we shall prove a factorization of the far-field operator.
\begin{thm}\label{Fac-F}
We have $F=\rot{L\mathcal{T}L^*}$, where $L=L_D^{(\hat{x})}$ is defined by
\be\label{def:L}
(Lu)(\tau)=\int_{t_{\min}}^{t_{\max}}\int_D e^{-{\rm i} \tau (\hat{x}\cdot y\rot{-}t)} u(y,t)\,dy dt,\qquad \tau\in (0, K )
\en
for all $u\in X_D$,
and $\mathcal{T}: X_D\rightarrow X_D$ is a multiplication operator defined by
\be\label{oT}
(\mathcal{T}u)(y,t):=\rot{\frac{1}{\sqrt{2\pi}}}e^{-{\rm i} k_c (\hat{x}\cdot y\rot{-}t)}\,S(y, t)\,u(y, t) .
\en
\end{thm}
\begin{proof}
We first claim that the adjoint operator $L^*: L^2(0, K )\rightarrow X_D$ of $L$ can be expressed by
\be\label{L}
(L^*\phi)(y,t):=\int_{0}^{K }e^{{\rm i}\tau(\hat{x}\cdot y\rot{-}t)}\phi(\tau)\,d\tau,\quad \phi\in L^2(0, K ).
\en
Indeed, for $u\in X_D$ and $\phi \in L^2(0,K)$, it holds that
\ben
\langle Lu, \phi \rangle_{L^2(0,K )}&=& \int_{0}^{K }\left(  \int_{t_{\min}}^{t_{\max}}\int_D e^{-{\rm i} \tau (\hat{x}\cdot y\rot{-}t)} u(y, t)\,dy\,dt\right)\overline{\phi(\tau)} d\tau \\
&=& \int_{t_{\min}}^{t_{\max}} \int_D u(y, t) \left(\int_{0}^{K }  \overline{\phi(\tau)\, e^{i\tau(\hat{x}\cdot y\rot{-}t)}}d\tau\right)\,dy\,dt
=\langle u, L^* \phi\rangle_{X_D}
\enn
which implies \eqref{L}.  By the definition of $\mathcal{T}$,
\ben
(\mathcal{T}L^*\phi)(y,t)=\rot{\frac{1}{\sqrt{2\pi}}}e^{-{\rm i} k_c (\hat{x}\cdot y\rot{-}t)}\,S(y, t)\int_{0}^{K }e^{is(\hat{x}\cdot y\rot{-}t)}\phi(s)\,ds,\quad \phi\in L^2(0, K ).\enn
Hence, combining \eqref{fxt}, \eqref{def:L} and \eqref{def:F} yields
\ben
(L\mathcal{T}L^*\phi)(\tau)&=& \int_{t_{\min}}^{t_{\max}}\int_D e^{-{\rm i} \tau(\hat{x}\cdot y\rot{-}t)}
\left(\rot{\frac{1}{2\pi}}e^{-{\rm i} k_c (\hat{x}\cdot y\rot{-}t)}\,S(y, t)\int_{0}^{K }e^{{\rm i}s(\hat{x}\cdot y\rot{-}t)}\phi(s)
\,ds\right)dy\, dt\\
&=&\rot{\frac{1}{\sqrt{2\pi}}}\int_{0}^{K }\int_D  e^{-{\rm i} (k_c+\tau-s)\hat{x}\cdot y}\phi(s)\,\left(
 \int_{t_{\min}}^{t_{\max}} S(y, t)e^{\rot{{\rm i}} (k_c+\tau-s)t}\,dt
\right)\,dy\,ds\\&=&\rot{\frac{1}{\sqrt{2\pi}}}\int_{0}^{K }\int_D  e^{-{\rm i} (k_c+\tau-s)\hat{x}\cdot y}\phi(s)\,f(y, k_c+\tau-s)\,dy\,ds=
\rot {(F\phi)(\tau)}.
 \enn
 This proves the factorization $F=\rot{L\mathcal{T}L^*}$.
\end{proof}
%Define a subset of $L^2(D\times(0, K ))$ by
%\ben
%X_D:=\left\{f|_{D\times(0, K )}: \begin{array}{lll}
%f(x, z)\in L^\infty(\R^3\times \C^+),\;
% {\rm Supp}\,f(\cdot, z)=D\;\mbox{for any} \; z\in \C^+, \\ \liminf_{\eta\rightarrow+\infty} |f(x, i\eta)|\;\mbox{is positive in a neighboring area of $\partial D$ in $D$}
%\end{array}
%\right\}.
%\enn
The operator $L^{(\hat{x})}_{D}$ maps a \rot{time-dependent source function $S(x,t)$ supported on $\overline{D}\times\rot{[}t_{\min}, t_{\max}\rot{]}$ to multi-frequency far-field patterns at the observation direction $\hat{x}$, that is, $w^\infty(\hat{x}, \cdot)=(L^{(\hat{x})}_{D} S)(\cdot)$.}
 It will be referred to as the 'data-to-pattern' operator within this paper.
Denote by  $\range  (L^{(\hat{x})}_{D})$ the range of the operator $L_D^{(\hat{x})}$ (see \eqref{def:L}) acting on $X_D$.

 \begin{lem}\label{com}
The operator $L_D^{(\hat{x})}: X_D\rightarrow L^2(0, K )$ is compact with dense range.
\end{lem}
\begin{proof} For any $u\in X_D$, it holds that
 $L_D^{(\hat{x})}u\in H^1(0, K )$ by definition \eqref{def:L}. Since $H^1(0, K )$ is compactly embedded into
$L^2(0, K )$,  we get the compactness of $L_D^{(\hat{x})}$.
To prove the denseness of ${\rm Range}(L_D^{(\hat{x})})$ in $L^2(0, K )$, we only need to prove  the injectivity of $(L_D^{(\hat{x})})^*$. If $(L_D^{(\hat{x})})^*\phi=0$ for some $\phi\in L^2(0, K )$,
from \eqref{L} it follows that
\ben
\int_{0}^{K }e^{{\rm i}\tau(\hat{x}\cdot y\rot{-}t)}\phi(\tau)\,d\tau=0\quad\mbox{in}\quad X_D.
\enn
Denote by $\tilde{\phi}$ the extension of $\phi$ by zero from $(0, K )$ to $\R$.
The previous relation implies
\ben
0= \frac{1}{\sqrt {2\pi}} \int_{\R} e^{{\rm i}\tau(\hat{x}\cdot y\rot{-}t)}\tilde{\phi}(\tau)\,d\tau=(\mathcal{F}^{-1}\tilde{\phi})(\xi)
\enn
where $\xi=\hat{x}\cdot y\rot{-}t\in ( \inf(\hat{x}\cdot D)\rot{-t_{\max}}, \sup(\hat{x}\cdot D) \rot{-t_{\min}})$.
The analyticity of $(\mathcal{F}^{-1}\tilde{\phi})(\xi)$ in $\xi\in \R$ yields
 the identically vanishing of the inverse Fourier transform of $\tilde{\phi}$. This proves $\phi=0$ and thus the injectivity of $(L_D^{(\hat{x})})^*$.
\end{proof}
Now we want to connect the ranges of $F$ and $L$. The existing range identities (see e.g., \cite[Theorem 2.15]{KG08} and \cite[Theorems 4.1 and 4.4]{K2002})  are not applicable to our case, because the real part of  the middle operator $\mathcal{T}$ (see \eqref{oT}) cannot be decomposed into the sum of a positive and a compact operator. Nevertheless, the multiplication form of the middle operator motivates us to establish a new range identity. We first recall that, for a bounded operator $F: Y\rightarrow Y$ in a Hilbert space $Y$ its real and imaginary parts  are defined respectively by
%\be\label{RI}
$\real F=(F+F^*)/2,\quad \rot{\ima F=(F-F^*)/(2i)},$
%\en
which are both self-adjoint operators. Furthermore, by spectral representation we define the self-adjoint and positive operator $|\real F|$ as
\ben
|\real F|=\int_{\R} |\lambda|\, d E_\lambda,\qquad \mbox{if}\quad \real F=\int_{\R} \lambda\, d E_\lambda.
\enn
The self-adjoint and positive operator $|\ima F|$ can be defined analogously.
In this paper the operator $F_{\#}$ is defined as
$F_{\#}:=|\real F| +|\ima F|.$
\rot{If} $F_{\#}$ is self-adjoint and positive, its square root $F_{\#}^{1/2}$ is well defined.
%\ben
%F_{\#}^{1/2}:=\int_{\R_+} \sqrt{\lambda}\, d E_\lambda,\qquad \mbox{if}\quad  F_{\#}=\int_{\R_+} \lambda\, d E_\lambda.
%\enn
We need the following auxiliary result from functional analysis.
\begin{thm}\label{range}
%Let $X\subset U\subset X^*$ be a Gelfand triple with a Hilbert space $U$ and a reflexive Bananch space $X$ such that the embeddings are dense.
Let $X$ and $Y$ be Hilbert spaces  and let $F: Y\rightarrow Y$, $L: X\rightarrow Y$, and $\mathcal{T}: X\rightarrow X$ be bounded linear operators such that $F=L\mathcal{T}L^*$. We make the following assumptions
\begin{itemize}
\item[(i)] $L$ is compact with dense range and thus $L^*$ is compact and one-to-one.
\item[(ii)] $\real \mathcal{T}$ and $\ima \mathcal{T}$ are both one-to-one and the operator $\mathcal{T}_{\#}=|\real \mathcal{T}| +|\ima \mathcal{T}|: X\rightarrow X$ is coercive, i.e., there exists $c>0$ with
\ben
\big\langle \mathcal{T}_{\#}\, \varphi, \varphi\big\rangle_X\geq c\,||\varphi||_X^2\quad\mbox{for all}\quad \varphi\in X.
\enn
\end{itemize}
Then the operator $F_{\#}$ is positive and  the ranges of $F_{\#}^{1/2}:Y\rightarrow Y$ and  $L: X\rightarrow Y$ coincide.
\end{thm}

\begin{proof}
We first recall from Part A in the proof of \cite[Theorem 2.15]{KG08} that it is sufficient to assume that $L^*: Y\rightarrow X$ has dense range in $X$. If otherwise, we may replace $X$ by the closed subspace $\overline{\mbox{Range} (L^*)}$ by using the orthogonal projection from $X$ onto $\overline{\mbox{Range} (L^*)}$. Below we shall prove the decomposition $F_{\#}=L\mathcal{T}_{\#}L^*$. For this purpose we only need to show
\be\label{rF}
|\real F|=L\,|\real \mathcal{T}|\, L^*,\qquad |\ima F|=L\,|\ima \mathcal{T}|\, L^*.
\en
It suffices to consider the real part of $F$, because the imaginary part can be treated similarly.
Since $\real F=L\,(\real \mathcal{T})\, L^*$ is self-adjoint, it has a complete orthonormal eigensystem $\{(\lambda_j, \psi_j): j\in \N\}$. Hence, the space $Y$ can be split into two closed orthogonal subspaces $Y=Y^-\oplus Y^+$ with
\ben
Y^-=\mbox{span}\{\psi_j: \lambda_j\leq 0\},\,\quad Y^+=\mbox{span}\{\psi_j: \lambda_j\geq 0\}.
\enn
\rot{We note that $Y^- \cap Y^+=\{0\}$, since $\lambda_j=0$ yields $\psi_j=0$}.
It is obvious that $\langle(\real F) \psi, \psi\rangle_Y$ is non-negative on $Y^+$ and non-positive on $Y^-$. \\ Consequently, $\langle (\real \mathcal{T})\phi,  \phi \rangle_X$ is non-negative on $U^+:=\overline{L^*(Y^+)}$ and is  non-positive on $U^-:=\overline{L^*(Y^-)}$, because
\be\label{rT}
\langle (\real \mathcal{T})\phi, \phi \rangle_X=\langle (\real \mathcal{T})L^*(\psi^\pm), L^*(\psi^\pm) \rangle_X
=\langle (\real F)(\psi^\pm), \psi^\pm \rangle_Y \rot{\lesseqqgtr} 0,
\en
where $\phi=L^*(\psi^\pm)$ with $\psi^\pm\in Y^\pm$. This implies that
\be\label{T}
\langle (\real \mathcal{T})\phi, \phi \rangle_X=0,\qquad\mbox{if}\quad \phi\in U^+ \cap U^-.
\en
For $\phi\in U^+ \cap U^-$, we have $\phi^\pm+t\phi \in U^\pm$ for all $\phi^\pm\in U^\pm$  and for all $t\in \C$. This together with the relation \eqref{T}
leads to $\langle (\real \mathcal{T})\phi, \phi^\pm \rangle_X=0$; see the Part C in the proof of \cite[Theorem 2.15]{KG08}. From this we deduce that  $(\real \mathcal{T})\phi=0$. Since $\real \mathcal{T}$ is one-to-one, we thus obtain $\phi=0$. In view of the denseness of the range of
 $L^*$, this proves the orthogonal decomposition  $X=U^+ \oplus U^- $.

To proceed with the proof, we denote by $P_Y^\pm$ the orthogonal projectors from $Y$ onto $Y^\pm$. Since $\real F$ is invariant on both $Y^+$ and $Y^-$, there holds
\be\label{realF}
|\real F|=(P^+-P^-) (\real F)= (\real F)(P^+-P^-)=L\,(\real \mathcal{T})\, L^*\,(P^+-P^-).
\en
Introduce the \rot{orthogonal projections}  $Q_U^\pm: X\rightarrow U^\pm$. It is then easy to conclude the relations
%\ben
$(\real \mathcal{T})\, L^*\, P_Y^\pm=(\real \mathcal{T})\,Q_U^\pm L^*$.
%\enn
Therefore, using \eqref{realF} and \eqref{rT},
\ben
|\real F|=L\,(\real \mathcal{T})\, L^*\,(P_Y^+-P_Y^-)=L\,[(\real \mathcal{T})(Q_U^+-Q_U^-)]\, L^*=L\, |\real \mathcal{T}|\, L^*,
\enn
which proves the first relation in \eqref{rF} and thus also the decomposition $F_{\#}=L\mathcal{T}_{\#}L^*$. By the second assumption (ii), we obtain the positivity of $F_{\#}$.
By the coercivity of  $\mathcal{T}_{\#}$, we can define the square root operator $\mathcal{T}_{\#}^{1/2}$, which is also coercive and self-adjoint.  Thus, we have a decomposition of $F_{\#}$ in the form
\ben
F_{\#}=(L\mathcal{T}_{\#}^{1/2})\, (L\mathcal{T}_{\#}^{1/2})^*=F_{\#}^{1/2}\,(F_{\#}^{1/2})^*.
\enn
Application of \cite[Theorem 1.21]{KG08} gives
%\ben
%\mbox{Range}\,
$(F_{\#}^{1/2})=\mbox{Range}\,(L).$
%\enn
%{\bf Step 1:} Reduction to the case $L: X\rightarrow Y$ is one-to-one. %Since the middle operator $T$ is an isomorphism, the factorization form of $F$ can be written as
%\ben
%S=[LT^*]\,(T^*)^{-1} [LT^*]^*=\tilde{L}\tilde{T}\tilde{L}^*
%\enn
%with \tilde{L}:
\end{proof}

To apply Theorem \ref{range} to our inverse problem, we set
%\ben
$F=F_D^{(\hat{x})}, \quad X=X_D,\quad Y=L^2(0, K )$
%\enn
and let $\mathcal{T}$ be the multiplication operator of \eqref{oT}. Since the source function $S(x, t)$ is real valued, it follows from \eqref{oT} that
\ben
(\real \mathcal{T})\, u=\rot{\frac{1}{\sqrt{2\pi}}}\cos \big(k_c (\hat{x}\cdot y\rot{-}t)\big)\, S(y, t)\, u(y,t),\\
(\ima \mathcal{T})\, u=\rot{\frac{1}{\sqrt{2\pi}}}\sin \big(k_c (\hat{x}\cdot y\rot{-}t)\big)\, S(y, t)\, u(y,t),
\enn both of them are one-to-one operators from $X_D$ onto $X_D$. The coercivity assumption on $S$ yields the  coercivity of  $\mathcal{T}_{\#}$, that is,
\ben
\langle \mathcal{T}_\#\, u, u\rangle&=&\int_{t_{\min}}^{t_{\max}}\int_D  \rot{\frac{1}{\sqrt{2\pi}}}\Big(|\cos \big(k_c (\hat{x}\cdot y\rot{-}t)\big)|+ |\sin \big(k_c (\hat{x}\cdot y\rot{-}t)\big)|\Big)\, S(y, t)\, |u(y,t)|^2\,dydt\\
&\geq& \rot{\frac{1}{\sqrt{2\pi}}}\int_{t_{\min}}^{t_{\max}}\int_D S(y, t)\, |u(y,t)|^2\,dydt
\geq
\rot{\frac{1}{\sqrt{2\pi}}}c_0\,||u||_{X_D}^2.
\enn

As a consequence of Theorem \ref{range}, we obtain
\be\label{RI}
\mbox{Range}\, [(F_D^{(\hat{x})})_{\#}^{1/2}]=\mbox{Range}\,(L_D^{(\hat{x})})\quad\mbox{ for all}\quad \hat{x}\in \s^2.
\en

%Below we describe the factorization method for imaging $D$.
Let $\chi(k)\in L^2(0, K )$ be some test function.
Denote by $(\lambda_n^{(\hat{x})}, \psi_n^{(\hat{x})})$ an eigensystem of the positive and self-adjoint operator $ (F_D^{(\hat{x})})_{\#}$, which is uniquely determined by the multi-frequency far-field patterns $\{w^\infty(\hat{x}, k): k\in (k_{\min}, k_{\max})\}$. Applying Picard's theorem and Theorem \ref{range}, we obtain
\be\label{indicator}
\chi \in \range (L_D^{(\hat{x})})\quad\mbox{if and only if}\quad \sum_{n=1}^\infty\frac{|\langle \chi, \psi_n^{(\hat{x})} \rangle|^2}{ |\lambda_n^{(\hat{x})}|}<\infty.
\en
To establish the factorization method, we now need to choose a suitable class of test functions
which usually rely on a sample variable in $\R^3$. The  inclusion relationship between the test function and ${\rm Range}(L_D^{(\hat{x})})$ should be associated with the inclusion relationship between the corresponding sample variable and the region $D$.

\begin{rem}\label{re3.1} In the special case that $k_{\min}=0$, we can also apply the range identity of \cite[Theorem 4.1]{K2002} to get \eqref{RI}. In fact,
since $w^\infty(\hat{x}, -k)=\overline{w^\infty(\hat{x}, k)}$, we may extend the bandwidth from $(0, k_{\max})$ to  $(-k_{\max}, k_{\max})$. Hence, one deduces from these new measurement data with $k_{\min}=-k_{\max}$ that
 $k_c=0$ and $K =k_{\max}$. Consequently, the middle operator $\mathcal{T}$ is self-adjoint, due to the multiplication form
$\mathcal{T} u= S\, u$ for $u\in X_D$ . This implies that $F_D^{(\hat{x})}$ is
also self-adjoint. Moreover, $F_D^{(\hat{x})}$ and $\mathcal{T}$ are both positive definite under the assumption \eqref{F} and thus  $(F_D^{(\hat{x})})_{\#}=F_D^{(\hat{x})}$, $\mathcal{T}_{\#}=\mathcal{T}$. The range identity stated in Theorem \ref{range} allows us to handle a more general class of wave-number bands, in particular an interval of wave-numbers bounded away from zero. %This has also relaxed the additional constraints of \cite[Section 6.2]{GS17} on the central frequency when $k_c\neq 0$ and if the source term does not rely on the wavenumber.
\end{rem}

\section{Range of $L_D^{(\hat{x})}$ and test functions}\label{sec3}
In this section we choose a proper class of test functions to characterize the range of $L_D^{(\hat{x})}$. Throughout the paper we set $T:=t_{\max}-t_{\rot{min}}>0$.
\begin{lem}\label{lem:range}
 Let $D_1, D_2 \subset \R^3$ be bounded domains such that $\hat{x}\cdot D_1\cap  \hat{x}\cdot D_2=\emptyset$.  Suppose one of the following relations holds \be\label{condition-T}
\inf(\hat{x}\cdot D_1)-\sup(\hat{x}\cdot D_2)>T,\quad  \inf(\hat{x}\cdot D_2)-\sup(\hat{x}\cdot D_1)>T.
\en
Then $\range  (L^{(\hat{x})}_{D_1})\cap \range  (L^{(\hat{x})}_{D_2})=\{0\}$, that is, the ranges of $L_{D_j}^{(\hat{x})}$ over $X_{D_1}$ and $X_{D_2}$ have trivial intersections.
\end{lem}
\begin{proof}
Let $f_j\in X_{D_j}$ be such that $L_{D_1}^{(\hat{x})} f_1=  L_{D_2}^{(\hat{x})} f_2:=\mathcal{G}(\cdot,\hat{x})$. By the definition of $L_D^{(\hat{x})}$ (see \eqref{def:L}), the function
\ben
\tau\mapsto \mathcal{G}(\tau,\hat{x})= \int_{t_{\min}}^{t_{\max}}\int_{D_1} e^{-{\rm i} \tau (\hat{x}\cdot y\rot{-}t)} f_1(y,t)\,dydt
= \int_{t_{\min}}^{t_{\max}}\int_{D_2} e^{-{\rm i} \tau (\hat{x}\cdot y\rot{-}t)} f_2(y,t)\,dydt
\enn
belongs to $L^2(0, K )$. Since $\mathcal{G}(\cdot,\hat{x})$ is analytic,
 the previous relation is valid for all $\tau\in \R$.  Extending $f_j$ by zero from $(t_{\min},t_{\max})$ to $\R$ and letting $\xi=\hat{x}\cdot y\rot{-}t$,  we can rewrite
the integrals over $D_j$ as
\ben
\int_{D_j} e^{-{\rm i} \tau (\hat{x}\cdot y\rot{-}t)} f_j(y,t)\,dy=\int_{\R}e^{-{\rm i} \tau \xi} \int_{\Gamma_j(\xi\rot{+}t, \hat{x})} f_j(y, t)ds(y)\,d\xi,
\enn
where $\Gamma_j(t,\hat{x})\subset D_j$ is defined as
\ben
\Gamma_j(t, \hat{x}):=\{y\in D_j: \hat{x}\cdot y=t\}\subset \R^3,\quad t\in \R,\quad j=1,2.
\enn
This implies that the function $\mathcal{G}(\cdot, \hat{x})$ is \rot{the  Fourier transform of $g_j$:}
\be\label{G}
\mathcal{G}(\tau, \hat{x})=\int_\R e^{-{\rm i} \tau \xi} g_j(\xi, \hat{x})\,d\xi,\quad\quad \tau\in \R,
\en
with
\ben g_j(\xi, \hat{x})&:=&\int_{t_{\min}}^{t_{\max}} \int_{\Gamma_j(\xi\rot{+}t, \hat{x})} f_j(y, t)ds(y) dt
= \int^{\xi\rot{+t_{\max}}}_{\xi\rot{+t_{\min}}} \int_{\Gamma_j(t, \hat{x})} f_j(y, \rot{t-\xi})ds(y) dt
\enn
for $j=1,2$.
%\ben
%\int_{\xi-T}^{\xi}\int_{\Gamma_j(t, \hat{x})} f_j(y, \xi-t)\,ds(y)\,dt=\int^{\xi}_{-\infty}\int_{\Gamma_j(t, \hat{x})} f_j(y, \xi-t)\,ds(y)\,dt
%\enn
By the arbitrariness of $\tau\in \R$, we get from \eqref{G} that $g_1(\xi,\hat{x})=g_2(\xi, \hat{x})$ for all $\xi\in \R$. On the other hand, observing that   $$\G_j(t, \hat{x})=\emptyset\qquad\mbox{if}\quad
 t>\sup(\hat{x}\cdot D_j)\quad\mbox{or}\quad t<\inf(\hat{x}\cdot D_j),$$ we have
$$g_j(\xi, \hat{x})=0\qquad\mbox{if}\quad \xi\rot{+}t_{\min}\rot{>\sup} (\hat{x}\cdot D_j)\quad\mbox{or} \quad\xi\rot{+}t_{\max}\rot{<\inf} (\hat{x}\cdot D_j).$$ This implies
\ben
{\rm supp}\, g_j(\cdot, \hat{x})\subset \Big(\,\inf(\hat{x}\cdot D_j)\rot{-t_{\max}},\, \sup (\hat{x}\cdot D_j)\rot{-t_{\min}}\,\Big),\quad j=1,2,
\enn
By the conditions in \eqref{condition-T}, it is clear that  one of the following relations holds:
\ben
\inf(\hat{x}\cdot D_1)\rot{-t_{\max}}>\sup(\hat{x}\cdot D_2)\rot{-t_{\min}},\\
    \inf(\hat{x}\cdot D_2)\rot{-t_{\max}}>\sup(\hat{x}\cdot D_1)\rot{-t_{\min}},
\enn
leading to $g_1(\xi,\hat{x})=g_2(\xi, \hat{x})\equiv 0$ for all $\xi\in \R$ in any case. Recalling \eqref{G}, we obtain
$L_{D_1}^{(\hat{x})} f_1=  L_{D_2}^{(\hat{x})} f_2\equiv 0$.
\end{proof}
As a consequence of the proof of Lemma \ref{lem:range}, we can get information on the supporting interval of the inverse Fourier transform of $L_D^{(\hat{x})} f$ for $f\in X_D$ as follows.
\begin{cor}\label{lem:supp} Let $D\subset \R^3$ be a bounded domain and $t_{\max}>t_{\min}$. Define
\ben
 \mathcal{G}(\tau,\hat{x}):= \int_{t_{\min}}^{t_{\max}}\int_{D} e^{-{\rm i} \tau (\hat{x}\cdot y\rot{-}t)} f(y,t)\,dydt,\qquad f\in X_D.
\enn
Then the support of the inverse Fourier transform of  $\mathcal{G}(\cdot,\hat{x})$ is contained in the interval
$\big( \inf(\hat{x}\cdot D)\rot{-t_{\max}},\; \sup(\hat{x}\cdot D)\rot{-t_{\min}} \big)$.
\end{cor}

For any $y\in \R^3$ and $\epsilon>0$, define  the test function $\phi^{(\hat{x})}_{y, \epsilon}\in L^2(0, K )$  by
\be\label{Test}
\phi^{(\hat{x})}_{y, \epsilon}(k)=\frac{1}{T \,|B_\epsilon(y)|}\int_{t_{\min}}^{t_{\max}}\int_{B_\epsilon(y)}e^{-{\rm i} k  (\hat{x}\cdot z\rot{-}t)} dzdt,\qquad k\in(0, K ).
\en
where $|B_\epsilon(y)|=4/3\pi \epsilon^3$ denotes the volume of the ball $B_\epsilon(y)\subset \R^3$. As $\epsilon\rightarrow 0$,  there holds the convergence
\be\label{test}
\phi^{(\hat{x})}_{y, \epsilon}(k)\rightarrow
 \phi^{(\hat{x})}_{y}(k):=\frac{1}{T }\left(\int_{t_{\min}}^{t_{\max}}e^{\rot{{\rm i} k t}} dt\right) \;e^{-{\rm i} k  \hat{x}\cdot y}
%=\frac{{\rm i} }{k T } \left( e^{-{\rm i} k  t_{\max}}- e^{-{\rm i} k  t_{\min}}\right)\,e^{-{\rm i} k  \hat{x}\cdot y}.
\en

%In the special case that $\epsilon=0$ and $t_{\min}=t_{\max}=0$, the function $\phi^{(\hat{x})}_{y, \epsilon}$ becomes $e^{-{\rm i} k \hat{x}\cdot y}$.
Below we describe the supporting interval of the inverse Fourier transform of the test functions defined by \eqref{Test}.
\begin{lem} For $\epsilon>0$,
we have
\be\label{a1}
&&[\mathcal{F}^{-1} \phi^{(\hat{x})}_{y, \epsilon}](\xi)>0\quad \mbox{if}\quad \xi\in (\hat{x}\cdot y-\epsilon\rot{-t_{\max}},\; \hat{x}\cdot y+\epsilon\rot{-t_{\min}}),\\ \label{a2}
&&[\mathcal{F}^{-1} \phi^{(\hat{x})}_{y, \epsilon}](\xi)=0\quad \mbox{if}\quad \xi\notin (\hat{x}\cdot y-\epsilon\rot{-t_{\max}},\; \hat{x}\cdot y+\epsilon\rot{-t_{\min}}).
\en %If $\epsilon=0$, it holds that
%\be \label{a3}
%&&[\mathcal{F}^{-1} \phi^{(\hat{x})}_{y}](\xi)=
%\left\{\begin{array}{lll}
%\sqrt{2\pi}/T \quad &&\mbox{if}\quad \xi\in (\hat{x}\cdot y+t_{\min},\; \hat{x}\cdot y+t_{\max}), \\
%0 \quad&&\mbox{if  otherwise}.
%\end{array}\right.
%\en
\end{lem}
\begin{proof} As done in \eqref{G},  we can rewrite  $\phi^{(\hat{x})}_{y, \epsilon}$ as \rot{the Fourier transform of $g_\epsilon(\xi,\hat{x})$:}
\ben
\phi^{(\hat{x})}_{y, \epsilon}(\tau)=\int_{\R} e^{-{\rm i} \tau\xi}g_\epsilon(\xi, \hat{x})\,d\xi,
\quad
g_\epsilon(\xi,\hat{x})=\frac{1}{T |B_{\epsilon}(y)|}\int_{\xi\rot{+t_{\min}}}^{\xi \rot{+t_{\max}}} \int_{\Gamma(t, \hat{x})} ds(z)dt,
\enn
with $\G(t, \hat{x})=\{z\in B_\epsilon(y): \hat{x}\cdot z=t\}$. Hence, $\mathcal{F}^{-1} \phi^{(\hat{x})}_{y, \epsilon}=\sqrt{2\pi}\;g_\epsilon(\cdot,\hat{x})$.
Observing that
$$\sup\big(\hat{x}\cdot  B_\epsilon(y)\big)=\hat{x}\cdot y+\epsilon,\quad  \inf\big(\hat{x}\cdot  B_\epsilon(y)\big)=\hat{x}\cdot y-\epsilon,$$ we obtain $\eqref{a1}$ and $\eqref{a2}$ from the expression of $g_{\epsilon}(\cdot, \hat{x})$. %If $\epsilon=0$, there holds
%\ben
%\phi^{(\hat{x})}_{y}(k)=\frac{1}{T }\int_{t_{\min}}^{t_{\max}} e^{-{\rm i} k  (t+\hat{x}\cdot y)}\,dt
%=\frac{1}{T }\int_{\R} e^{-{\rm i} k \xi} g(\xi)\,d\xi,
%\enn
%where
%\ben
%g(\xi):=\left\{\begin{array}{lll}
%1\quad &&\mbox{if}\quad \xi\in (\hat{x}\cdot y+t_{\min},\; \hat{x}\cdot y+t_{\max}). \\
%0 \quad&&\mbox{if otherwise}.
%\end{array}\right.
%\enn
%Therefore, $[\mathcal{F}^{-1} \phi^{(\hat{x})}_{y}](\xi)=\sqrt{2\pi}\,g(\xi)/T $.
\end{proof}
Introduce the strip (see Figure \ref{strip})
\be\label{K}
K_D^{(\hat{x})}:=\{y\in \R^3: \inf(\hat{x}\cdot D)<\hat{x}\cdot y<\sup (\hat{x}\cdot D)\}\subset \R^3.
\en
\begin{figure}[htbp]
  \centering
    \includegraphics[scale=0.6]{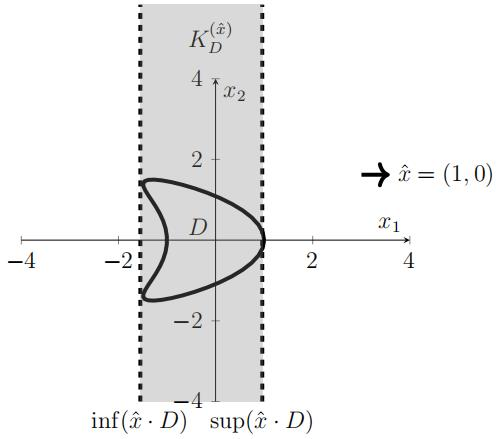}\caption{Illustration of the strip $K_D^{(\hat{x})}$ with $\hat{x}=(1,0)$.}
\label{strip}
\end{figure}

The set $K_D^{(\hat{x})}\subset \R^3$ represents the smallest strip containing $D$ and perpendicular to the vector $\hat{x}\in\s^2$. We shall establish a computational criterion for imaging $K_D^{(\hat{x})}$ from the multi-frequency far-field data $u^\infty(\hat{x},k)$ with $k\in(k_{\min}, k_{\max})$.
\begin{lem}\label{lem3.4}
(i) For $y\in K_D^{(\hat{x})}$,  there exists $\epsilon_0>0$ such that  $\phi^{(\hat{x})}_{y, \epsilon}\in \range (L_D^{(\hat{x})})$ for all $\epsilon\in(0, \epsilon_0)$.
(ii) If $y\notin K_D^{(\hat{x})}$, we have
 $\phi^{(\hat{x})}_{y, \epsilon}\notin \range (L_D^{(\hat{x})})$ for all $\epsilon>0$.
%\item[(iii)] \tcr{For $\epsilon=0$, we have
% $\phi^{(\hat{x})}_{y}\in \range (L_D^{(\hat{x})})$ if and only if \rot{$y\in K_D^{(\hat{x})}$.} }

\end{lem}
\begin{proof}
($i$) If $\hat{x}\cdot y\in \hat{x}\cdot D$, there must exist some $z\in D$ and $\epsilon_0>0$ such that  $\hat{x}\cdot y=\hat{x}\cdot z$ and  $B_\epsilon(z)\subset D$ for all $\epsilon\in(0, \epsilon_0)$.  Moreover, we have
$\phi^{(\hat{x})}_{y, \epsilon}=\phi^{(\hat{x})}_{z, \epsilon}$.
Setting
\ben
u(x, t):=\left\{\begin{array}{lll}
\frac{1}{|B_\epsilon(z)|\;T },\quad &&\mbox{if}\; x\in B_\epsilon(z),  t\in (t_{\min}, t_{\max}),\\
0\quad&&\mbox{if otherwise}.
\end{array}\right.
\enn
It is obvious that $u(x,t)\in L^2(D\times (t_{\min}, t_{\max})$. By the definition of $L_D^{(\hat{x})}$ (see \eqref{def:L}), it is easy to see
 $\phi^{(\hat{x})}_{z, \epsilon}=L_D^{(\hat{x})}u$.

($ii$) Given $y\notin K_D^{(\hat{x})}$, we suppose on the contrary that
$\phi^{(\hat{x})}_{y, \epsilon}= L_D^{(\hat{x})}f$
with some $f\in L^2(D\times(t_{\min}, t_{\max}))$, i.e.,
\be\label{eq}
\phi^{(\hat{x})}_{y, \epsilon}(\tau)=\int_{t_{\min}}^{t_{\max}}\int_{D} e^{-{\rm i} \tau (\hat{x}\cdot z\rot{-}t)} f(z,t)\,dz dt,\qquad \tau\in (0, K ).
\en
By the analyticity in $\tau$, the above relation \rot{can be extended to $\tau\in \R$}. Hence, the supporting intervals of the inverse Fourier transform of both sides of \eqref{eq} must coincide.
Using \eqref{a1} and Corollary \ref{lem:supp}, we obtain
\ben
(\hat{x}\cdot y-\epsilon\rot{-t_{\max}},\; \hat{x}\cdot y+\epsilon\rot{-t_{\min}})\, \subset\,\left( \inf(\hat{x}\cdot D)\rot{-t_{\max}},\; \sup(\hat{x}\cdot D)\rot{-t_{\min}}  \right),
\enn
leading to \be\label{eq:ep}
\inf(\hat{x}\cdot D)+\epsilon\leq \hat{x}\cdot y \leq \sup(\hat{x}\cdot D)-\epsilon,\quad\mbox{for all}\; \epsilon>0.
\en
This implies that $y\in K_D^{(\hat{x})}$,
a contradiction to the assumption $y\notin K_D^{(\hat{x})}$. This proves  $\phi^{(\hat{x})}_{y, \epsilon}\notin \range (L_D^{(\hat{x})})$ for all $\epsilon>0$.
\end{proof}

\section{Indicator functions with multi-frequency far/near-field data}\label{sec4}
By \\
Lemma \ref{lem3.4}, the functions $\phi^{(\hat{x})}_{y, \epsilon}$ with a small $\epsilon>0$ can be taken as test functions to characterize $D$ through \eqref{indicator}.
We first consider the indicator function involving $\phi^{(\hat{x})}_{y, \epsilon}\rot{:}$
\be\label{indicator4}
W_\epsilon^{(\hat{x})}(y):=\left[\sum_{n=1}^\infty\frac{|\langle \phi^{(\hat{x})}_{y, \epsilon}, \psi_n^{(\hat{x})} \rangle|_{L^2(0, K )}^2}{ |\lambda_n^{(\hat{x})}|}\right]^{-1}, \; y\in \R^3.
\en
The analogue of \eqref{indicator4} was used in \cite{GS} for imaging the support of a wave-number-independent source function.
Combining Theorem \ref{range} and Lemma \ref{lem3.4} yields
\begin{theorem}\label{Th:factorization}
(i) If $y\in K_D^{(\hat{x})}$, there exists $\epsilon_0>0$ such that
$W_\epsilon^{(\hat{x})}(y)>0$ for all $\epsilon\in(0, \epsilon_0)$.
(ii)If $y\notin K_D^{(\hat{x})}$  there holds
$W_\epsilon^{(\hat{x})}(y)=0$ for all $\epsilon>0$.
\end{theorem}
Since $\phi^{(\hat{x})}_{y, \epsilon}$ convergences uniformly to $\phi^{(\hat{x})}_{y}$ over the finite wave-number interval $[k_{\min}, \\
k_{\max}]$, we shall use the limiting function $\phi^{(\hat{x})}_{y}$ in place of $\phi^{(\hat{x})}_{y, \epsilon}$ in the aforementioned indicator function. Consequently, we
introduce a new indicator function
\be\label{indicator1}
W^{(\hat{x})}(y):=\left[\sum_{n=1}^\infty\frac{|\langle \phi^{(\hat{x})}_{y}, \psi_n^{(\hat{x})} \rangle|_{L^2(0, K )}^2}{ |\lambda_n^{(\hat{x})}|}\right]^{-1}%\nonumber
\sim
\left[\sum_{n=1}^N\frac{|\langle \phi^{(\hat{x})}_{y}, \psi_n^{(\hat{x})} \rangle|_{L^2(0, K )}^2}{ |\lambda_n^{(\hat{x})}|}\right]^{-1}\en
for $y\in \R^3$,
where the integer $N\in \N$ is a truncation number.  Taking the limit $\epsilon\rightarrow 0$ in Theorem \ref{Th:factorization}, it follows that
\be\label{W}
W^{(\hat{x})}(y)=\left\{\begin{array}{lll}
\geq 0\quad&&\mbox{if}\quad y\in K_D^{(\hat{x})},\\
0\quad&&\mbox{if}\quad y\notin K_D^{(\hat{x})}.
\end{array}\right.
\en
Hence, the values of $W^{(\hat{x})}$ in the strip \rot{$K_D^{(\hat{x})}$} should be relatively bigger than those elsewhere.  %\rot{The above formula describes a necessary and sufficient condition for characterizing the strip
%$K_D^{(\hat{x})}$. Hence,}
%the indicator $W^{(\hat{x})}$
%is the characteristic function over the domain $K_D^{(\hat{x})}$.
 In the case of sparse observation directions $\{\hat{x}_j: j=1,2,\cdots, M\}$, we shall make use of the following indicator function:
\be\label{W}
W(y)\rot{= \left[\sum_{j=1}^M \frac{1}{W^{(\hat{x}_j)}(y) }\right]^{-1}}= \left[\sum_{j=1}^M\sum_{n=1}^N\frac{|\langle \phi^{(\hat{x}_j)}_{y}, \psi_n^{(\hat{x}_j)} \rangle|_{L^2(0, K )}^2}{ |\lambda_n^{(\hat{x}_j)}|}\right]^{-1}, \; y\in \R^3.
\en
Define the $\Theta$-convex hull of $D$ associated with $\{\hat{x}_j: j=1,2,\cdots, M\}$ as
%\ben
$\Theta_D:=\bigcap_{j=1,2,\cdots, M} K_D^{(\hat{x}_j)}$.
%\enn

\begin{theorem}\label{TH-hull}
We have $W(y)\geq0$ if $y\in \Theta_D$ and $W(y)=0$ if
$y\notin \Theta_D$.
%\begin{itemize}
%\item[(i)] If $y\in \Theta_D$, it holds that $W(y)>0$.
%\item[(ii)] If $y\notin \Theta_D$,  it holds that  $W(y)=0$.
%\end{itemize}
\end{theorem}
\begin{proof}
If $y\in \Theta_D$, then $y\in K_D^{(\hat x_j)}$ for all $j=1,2,...,M$, yielding that $\hat x_j \cdot y \in \hat x_j \cdot D$. Hence, one deduces from  Theorem \ref{Th:factorization} that  $0\leq W^{(\hat x_j)}(y)<\infty$ for all $j=1,2,...,M$, implying that $0\leq W(y)$. On the other hand, if $y\notin \Theta_D$, there must exist some unit vector $\hat{x}_l$ such that $y\notin K_D^{(\hat x_l)}$. Again, using Theorem \ref{Th:factorization} we get $$
[W^{(\hat{x}_l)}(y)]^{-1}=\sum_{n=1}^\infty\frac{|\langle \phi^{(\hat{x}_l)}_{y}, \psi_n^{(\hat{x}_l)} \rangle|_{L^2(0, K )}^2}{ |\lambda_n^{(\hat{x}_l)}|}=\infty,$$ which proves $W(y)=0$ for $y\notin \Theta_D$.
\end{proof}
The values of $W(y)$ are expected to be large for $y\in \tcr{\Theta_D}$ and small for those $y\notin \Theta_D$. % because $\Theta_D\subset K_D^{(\hat{x}_j)}$ for all $j=1,2,\cdots, M$.
Below we shall provide a physical interpretation of the proposed inversion algorithm and build up connections with the time-dependent inverse source problems.   We first remark that, the above factorization method with multi-frequency data carries over to near-field measurements in three dimensions.  More precisely,
the proposed factorization method can be slightly modified to get an image of  the annular region
\be\label{K}
\tilde{K}_D^{(x)}:=\{y\in \R^3: \inf_{z\in D} |x-z|<|x-y|<\sup_{z\in D} |x-z|\}\subset \R^3
\en
for every fixed measurement position $|x|=R$. For this purpose we suppose $D\subset B_R$ for some $R>0$ and define the near-field operator
$\mathcal{N}^{(x)}:  L^2(0, K )\rightarrow L^2(0, K )$ by
\be \label{def:N}
&&(\mathcal{N}^{(x)}\phi)(\tau):=\int_{0}^{K } w(x, k_c+\tau-s)\,\phi(s)\,ds,\quad \rot{\tau \in (0,K),} %\\ \nonumber %\quad \tau\in(0, K )
%&=&\int_{0}^{K } \int_{D} \frac{e^{i (k_c+\tau-s) |x-y|}}{4\pi |x-y|}\,f(y, k_c+\tau-s )dy\,\phi(s)\,ds \\ \nonumber
%&=&\int_{0}^{K } \int_{D} \frac{e^{i (k_c+\tau-s) |x-y|}}{4\pi |x-y|} \left(\int_{t_{\min}}^{t_{\max}}
%S(y, t)e^{-i (k_c+\tau-s)t} dt\right)\, dy\,\phi(s)\,ds
\en
where $x\in \partial B_R$ is a measurement position and $w\in H^2(B_R)$ is the solution to the Helmholtz equation \eqref{Helmholtz}. Following the proof of Theorem \ref{Fac-F}, we obtain a factorization of the near-field operator as %follows:
%\ben
$\mathcal{N}^{(x)}=\tilde{L} \tilde{T}\tilde{L}^*,$
%\enn
where
 $\tilde{L}=\tilde{L}_D^{(x)}: X_D\rightarrow L^2(0, K )$ is defined by
\ben
(\tilde{L}u)(\tau)=\int_{t_{\min}}^{t_{\max}}\int_D e^{\rot{i\tau (|x-y|+t)}} u(y,t)\,dy dt,\qquad \tau\in (0, K )
\enn
for all $u\in X_D$,
and the middle operator $\tilde{T}: X_D\rightarrow X_D$ is again a coercive multiplication operator defined by
\ben
(\tilde{T}u)(y,t):=\frac{e^{{\rm i} k _c (|x-y|\rot{+}t)}}{\rot{\sqrt{32\pi^3}} |x-y|}\,u(y, t)\, S(y, t),\qquad |x|=R.
\enn
%Note the the adjoint of $\tilde{L}: L^2(0, K)\rightarrow X_D$ is defined by
%\ben
%(\tilde{L}^* \phi)(y, t):=\int_{0}^K e^{is (t-|x-y|)} \phi(s)\,ds \in X_D,\quad \phi\in  L^2(0, K).
%\enn
Choose the test function
\ben
\tilde{\phi}^{(x)}_{y,\epsilon}(k):=\frac{1}{T\,|B_\epsilon(y)|}\int_{t_{\min}}^{t_{\max}}\int_{B_\epsilon(y)} \frac{e^{ik (|x-z|\rot{+}t)}}{4\pi |x-z|}\,dz\,dt,
\enn
which tends uniformly to
\ben
\tilde{\phi}^{(x)}_{y}(k):=\frac{\rot{-}i\,e^{{\rm i} k  |x-y|}}{4\pi k |x-y|T }\left(e^{\rot{{\rm i} k  t_{\max}}}-e^{\rot{{\rm i} k  t_{\min}}}\right),
\enn
as $\epsilon\rightarrow 0$ for all $k\in[k_{\min}, k_{\max}]$.
Introduce the indicator function
\be \label{near-indicator}
\widetilde{W}^{(x)}(y):=\left[\sum_{n=1}^\infty\frac{|\langle \tilde{\phi}^{(x)}_{y}, \tilde{\psi}_n^{(x)} \rangle|_{L^2(0, K )}^2}{ |\rot{\tilde{\lambda}}_n^{(x)}|}\right]^{-1}, \qquad y\in \R^3,
\en
where $(\tilde{\lambda}_n^{(x)}, \tilde{\psi}_n^{(x)})$ is an eigensystem of the near-field operator $(\mathcal{N}^{(x)})_\#$.
As the counterpart to  \eqref{W}, one can
 show in the near-field case that
\begin{cor}\label{indicator-near}
We have
$\mbox{Range}\, [(\mathcal{N}^{(x)})_{\#}^{1/2}]=\mbox{Range}\,(\tilde{L}^{(x)})$ for all $x\in \pa B_R$. The indicator $\widetilde{W}^{(x)}\geq 0$ in $\tilde{K}_D^{(x)}$ and vanishes identically in the exterior of $\tilde{K}_D^{(x)}$.
% \ben
%\widetilde{W}^{(x)}(y)=\left\{\begin{array}{lll}
%\geq 0\quad&&\mbox{if}\quad y\in \tilde{K}_D^{(x)},\\
%0\quad&&\mbox{if}\quad y\notin \tilde{K}_D^{(x)}.
%\end{array}\right.
%\enn
\end{cor}

Now we want to bridge this near-field indicator functional with the time-domain signals of wave equations.
In the near-field case, we suppose that the sparse  data are given by $\{ {w}(x_j, k): x_j\in \partial B_R, k\in(k_{\min}, k_{\max})\}$, which can be considered as the \rot{inverse} Fourier transform of the time-dependent data $\{U(x_j, t): x_j\in \partial B_R, t\in(0, T)\}$ for some $T>2R+t_{\max}$.
The time-dependent \rot{signal} $t\mapsto U(x, t)$ with a fixed $x\in \partial B_R$ has a compact support, because the source term is compactly supported in $\rot{\overline{D}}\times \rot{[}t_{\min}, t_{\max}\rot{]}$. Physically, this can be  explained by the Huygens principle in 3D.
Since the wave speed has been normalized to be one, it is not difficult to observe that the arrival time point $T_{arr}$ and the terminal time point $T_{ter}$ of the signal recorded by the sensor at $|x|=R$ are %(see Figure \ref{distance})
\ben
T_{arr}=t_{\min}+\inf_{z\in D} |x-z|,\qquad T_{ter}=t_{\max}+\sup_{z\in D} |x-z|,
\enn
respectively, where $\rot{[}t_{\min}, t_{\max}\rot{]}$ represents the duration period for source radiating. This explains why the minimum distance and maximum distance between a measurement position $x$ and our target $D$
 (which are equivalent to the annulus $\tilde{K}_D^{(x)}$)  can be recovered from the multi-frequency near-field data at $|x|=R$.

\section{Discussions on source functions with two disconnected supports}\label{mul-compo}
In this section we discuss the factorization method for imaging the support of a wave-number-dependent source term with two disconnected components. For simplicity we only consider the far-field measurement data at multi-frequencies.
Suppose that $D=D_1\cup D_2\subset \mathbb R^3$ contains two disjoint sub-domains $D_j$ ($j=1,2$) which can be separated by some plane. The indicator function (4.2) can be used to image the strip $K_D^{(\hat{x})}$ if the multi-frequency data are observed at the direction $\hat{x}$. With the data for all observation directions the convex hull of $D$  can be recovered from the indicator (4.4). Physically, it would be more interesting to determine $\mbox{ch}(D_j)$ for each $j=1,2$, whose union is usually only a subset of $\mbox{ch}(D)$.
Analogously to Lemma 3.3 and Theorem 4.1,  we can prove the following results.
% Lemma \ref{lem3.4} and Theorem \ref{Th:factorization},  we can prove the following results.
\begin{cor} Let $\hat{x}\in \s^2$ be fixed.
\begin{itemize}
\item[(i)] For $y\in{K_{D_1}^{(\hat{x})}}\cup {K_{D_2}^{(\hat{x})}} $,  we have $\phi^{(\hat{x})}_{y,\epsilon}\in \range (L_D^{(\hat{x})})$ for all $\epsilon\in(0, \epsilon_0)$ with some $\epsilon_0>0$.
\item[(ii)] If $y\notin {K_{D_1}^{(\hat{x})}}\cup {K_{D_2}^{(\hat{x})}} $, we have
$\phi^{(\hat{x})}_{y,\epsilon}\notin \range (L_D^{(\hat{x})})$ for all $\epsilon>0$,
 provided one of the following conditions holds
\be\label{ab}
\quad\quad \quad\;(a)\; \inf (\hat{x}\cdot D_2)-\sup (\hat{x}\cdot D_1)>T;\; (b) \; \inf (\hat{x}\cdot D_1)-\sup (\hat{x}\cdot D_2)>T .
\en
\item[(iii)] Let the indicator function $W^{(\hat{x})}$ be defined by (4.4). Under one of the conditions in \eqref{ab} it holds that
\ben
W^{(\hat{x})}(y)=\left\{\begin{array}{lll}
\geq 0\quad&&\mbox{if}\quad y\in  {K_{D_1}^{(\hat{x})}}\cup {K_{D_2}^{(\hat{x})}} ,\\
0\quad&&\mbox{if}\quad y\notin ({K_{D_1}^{(\hat{x})}}\cup {K_{D_2}^{(\hat{x})}} ).
\end{array}\right.
\enn
\end{itemize}
\end{cor}
Note that the conditions in \eqref{ab} imply that $(\hat{x}\cdot D_1)\cap (\hat{x}\cdot D_2)=\emptyset$. Furthermore, the inverse Fourier transform of $L_D^{(\hat{x})} f$ with $f\in X_D$ is supported in the following two disjoint intervals (see Lemma \ref{lem1} below for a detailed proof)
\ben
\Big(\inf(\hat{x}\cdot D_1)-t_{\max},\, \sup(\hat{x}\cdot D_1)-t_{\min}\Big)\; \bigcup \;\Big(\inf(\hat{x}\cdot D_2)-t_{\max},\, \sup(\hat{x}\cdot D_2)-t_{\min}\Big).
\enn
For small $T>0$, there exists at least one observation directions $\hat{x}\in \s^2$ such that the relations in \eqref{ab} hold, because $D_1$ and $D_2$ can be separated by some plane according to our assumption. If the conditions \eqref{ab} hold for all observation directions $\hat{x}\in \s^2$, one can make use of the indicator function (4.4) to get an image of the set
 $\bigcap_{j=1,2,\cdots M} \{ K_{D_1}^{(\hat{x}_j)} \cup K_{D_2}^{(\hat{x}_j)}\}$, which is usually larger than $\mbox{ch}(D_1)\cup \mbox{ch}(D_2)$.
This means that our approach can only be used to recover partial information of  $\mbox{ch}(D_j)$,
when the source radiating period $T$ is sufficiently small in comparision with the distance between $D_1$ and $D_2$.
The numerical experiments performed in Section 5 confirm the above theory; see Figures 8 and 9.

Physically, the conditions in \eqref{ab} ensure that the time-dependent signals recorded at $\hat{x}$ has two disconnected supports which correspond to the radiated wave fields from $D_1$ and $D_2$, respectively. If one can
split the multi-frequency far-field patterns at a single observation direction, it is still possible to recover $\{ K_{D_1}^{(\hat{x})} \cup K_{D_2}^{(\hat{x})}\}$ even if the conditions in \eqref{ab} cannot be fulfilled. Below we prove that the multi-frequency far-field patterns excited by two disconnected source terms can be split under additional assumptions.

To rigorously formulate the splitting problem, we go back to the wave equation (1.1),
where $D=D_1\cup D_2$ contains two disjoint bounded and connected sub-domains $D_1$ and $D_2$. As seen in (1.3), the solution $U=G*S$ can be written explicitly as the convolution of the fundamental solution $G$ with the source term $S$.
Define $U_j:=G*S_j$ with $S_j:=S|_{D_j\times \R_+}$. It is obvious that
$U(x,t)=U_1(x,t)+U_2(x,t)$ for $(x,t)\in \R^3\times \R_+$.
In the frequency domain, we consider the time-harmonic source problem
\ben
\Delta u+k^2 u=-f_1(\cdot,k)-f_2(\cdot,k)\quad \mbox{in}\quad \R^3\times \R_+,
\enn
where $\mbox{supp} f_j(\cdot ,k)=D_j$ for any $k>0$  and
\ben
f_j(x,k):=\int_{t_{\min}}^{t_{\max}} S_j(x, t) e^{ik t} dt,\quad x\in D_j.
\enn
Let $u_j$ be the unique radiating solution to
\ben
\Delta u_j+k^2 u_j=-f_j(x,k)\qquad\mbox{in}\quad \R^3\times \R_+,\quad j=1,2.
\enn
Denote by $u^\infty(\hat{x}, k), u_j^\infty(\hat{x}, k)$ the far-field patterns of $u$ and $u_j$ at some fixed observation direction $\hat{x}=x/|x|$, respectively.
It is obvious that $u^\infty=u_1^\infty+u_2^\infty$, where
\be\label{FF}
u_j^\infty(\hat{x}, k)=\int_{t_{\min}}^{t_{\max}}\int_{D_j} S_j(y, t) e^{-ik (\hat{x}\cdot y-t)}\,dy dt,\quad k\in \R_+.
\en
%Since $S_j$ are real-valued, it holds that $u_j^\infty(-k,\hat{x})=\overline{u_j^\infty(\hat{x}, k)}$. Hence, the far-field data set for positive wavenumbers can be analytically extended to negative wavenumbers.
The splitting problem in the frequency domain can be formulation as follows: Given a fixed observation direction $\hat{x}\in\s^2$,
split $\{u_j^\infty(\hat{x}, k): k\in \R\}$ from the data $\{u^\infty(\hat{x}, k):k\in \R\}$ for $j=1,2$.

Set  $l_j:=\sup(\hat{x}\cdot D_j)-\inf(\hat{x}\cdot D_j)$ and $\Lambda_j=T+\ell_j$.
We make the following assumptions on the time-dependent source function $S(x,t)$.
%\ben
\begin{itemize}
\item[(i)] $S(x,t)\geq c_0>0$ for $(x,t)\in D\times (t_{\min}, t_{\max})$.
\item[(ii)] $S$ is analytic on $\overline{D}\times [t_{\min}, t_{\max}]$ and the boundary $\partial D$ is analytic.
\item[(iii)] Either $\inf(\hat{x}\cdot D_1)< \inf(\hat{x}\cdot D_2)$, or $\sup(\hat{x}\cdot D_1)> \sup(\hat{x}\cdot D_2)$.
\end{itemize}
\begin{lem}\label{lem1} Under the assumption (i), the supporting interval of  $\mathcal{F}^{-1}(u_j^\infty)$ is $I_j:=(\inf(\hat{x}\cdot D_j)-t_{\max}, \sup(\hat{x}\cdot D_j)-t_{\min} ) $ and  the function $t\mapsto (\mathcal{F}^{-1}u_j^\infty)(t)$ is positive
in $I_j$. Moreover,  $(\mathcal{F}^{-1}u_j^\infty)(t)$ is  analytic  in $t\in I_j$ under the additional assumption (ii).
\end{lem}
\begin{proof} We carry out the proof following the ideas in the proof of Lemma 3.1.
The far-field expression \eqref{FF} can be rewritten as
\ben
u_j^\infty(\hat{x}, k)=\int_{\R}e^{-ik\xi}g_j(\xi)\,d\xi,\qquad j=1,2,
\enn
where
\be\label{g}
 \quad \quad \quad g_j(\xi)&:=&\int_{t_{\min}}^{t_{\max}} \int_{\Gamma_j(\xi+t)} S_j(y, t)ds(y) dt
= \int^{\xi+t_{\max}}_{\xi+t_{\min}} \int_{\Gamma_j(t)} S_j(y, t-\xi)ds(y) dt.
\en
Here $\Gamma_j(t)\subset D_j$ is defined as
$
\Gamma_j(t):=\{y\in D_j: \hat{x}\cdot y=t\}.
$
%Since $\G_j(t)=\emptyset$ for $t<\inf(\hat{x}\cdot D_j)$ or $t>\max(\hat{x}\cdot D_j)$, it is obvious that $g_j(\xi)=0$ for $\xi<\inf(\hat{x}\cdot D_j)+t_{\min}$ or $\xi>\sup(\hat{x}\cdot D_j)+t_{\max}$. On the other hand,
By the assumption (i) we deduce from \eqref{g} with $\xi=\inf(\hat{x}\cdot D_j)-t_{\max}+\epsilon$,  $\epsilon\in (0, \Lambda_j)$ that
\ben
g_j(\xi)&=&
 \int^{\inf(\hat{x}\cdot D_j)+\epsilon}_{\inf(\hat{x}\cdot D_j)+\epsilon-T} \int_{\Gamma_j(t)} S_j(y, t-\xi)ds(y) dt\\
 &=&\int^{\inf(\hat{x}\cdot D_j)+\epsilon}_{\inf(\hat{x}\cdot D_j)} \int_{\Gamma_j(t)} S_j(y, t-\xi)ds(y) dt\\
 &>& 0,
 \enn
because
\ben
&& t-\xi\in(t_{\max}-\epsilon, t_{\max})\quad\mbox{if}\quad t\in \Big(\inf(\hat{x}\cdot D_j), \inf(\hat{x}\cdot D_j)+\epsilon\Big),
%&&\rot{ t-\xi\in\Big(t_{\min}+t_{\max}-\inf(\hat{x}\cdot D_j)-\epsilon,  2t_{\max}-\inf(\hat{x}\cdot D_j)-\epsilon\Big)\quad\mbox{if}\quad t\in (t_{\min}, t_{\max}).}
\enn
and
\ben
\int_{\Gamma_j(t)} S_j(y, t-\xi)ds(y)=0\quad \mbox{if}\quad t<\inf(\hat{x}\cdot D_j)
\enn
due to the fact that $\Gamma_j(t)=\emptyset$.
Since $g_j$ coincides with the inverse Fourier transform of $u_j^\infty$ by a factor, this proves the first part of the lemma. The analyticity of  $(\mathcal{F}^{-1}u_j^\infty)(t)$ in $t\in I_j$ follows from \eqref{g} under the assumption (ii).
\end{proof}
Next we show that the multi-frequency far-field measurement data at a fixed observation direction can be uniquely split. Note that the splitting is obvious under the conditions in \eqref{ab}, because by Lemma \ref{lem1} the inverse Fourier transform of $u^\infty(\hat{x}, k)$ has two disconnected components.
\begin{thm}
Suppose that there are two time-dependent sources $S$ and $\tilde{S}$ with \mbox{supp}\;$\tilde{S}=\overline{\tilde{D}}\times [t_{\min}, t_{\max}]$ and $\tilde{D}=\tilde{D}_1\cup\tilde{D}_2$. Here the source function $\tilde{S}$ and its support $\overline{\tilde{D}}$ are also required to satisfy the assumptions $(i)$-$(iii)$. Let $\tilde{u}_j^\infty$ be defined by \eqref{FF} with $D_j, S_j$ replaced by $\tilde{D}_j$, $\tilde{S}_j:=\tilde{S}|_{\tilde{D}_j\times (t_{\min}, t_{\max})}$, respectively. Then the relation
\be\label{uinfinity}
\quad \quad \quad u^{\infty}(\hat{x}, k)=u_1^\infty(\hat{x}, k)+u_2^\infty(\hat{x}, k)=\tilde{u}_1^\infty(\hat{x}, k)+\tilde{u}_2^\infty(\hat{x}, k),\; k\in(k_{\min}, k_{\max})
\en implies that $u_j^\infty(\hat{x}, k)=\tilde{u}_j^\infty(\hat{x}, k)$ for $k\in(k_{\min}, k_{\max})$ and $j=1,2$.
\end{thm}
\begin{proof} By the analyticity of $u_j^\infty$, $\tilde{u}_j^\infty$ in $k\in \R$, the function $u^\infty$ can be analytically extended to the whole real axis.
Denote by $[T_{\min}, T_{\max}]$ the supporting interval of
the inverse Fourier transform of $u^\infty$ with respect to $k$. In view of assumption (iii) and  Lemma \ref{lem1},  we may suppose without loss of generality that $T_{\min}=\Lambda_{\min}-t_{\max}$ with
\ben
\Lambda_{\min}=\inf(\hat{x}\cdot D_1)=\inf(\hat{x}\cdot \tilde{D}_1)<
\inf(\hat{x}\cdot D_2),\quad \Lambda_{\min}<
\inf(\hat{x}\cdot \tilde{D}_2).
\enn
If otherwise, there must hold
$T_{\max}=\Lambda_{\max}-t_{\min}$ with
\ben
\Lambda_{\max}=\sup(\hat{x}\cdot D_1)=\sup(\hat{x}\cdot \tilde{D}_1)>
\sup(\hat{x}\cdot D_2),\quad \Lambda_{\max}>
\sup(\hat{x}\cdot \tilde{D}_2)
\enn
and the proof can be carried out similarly.

Define $w_j=u_j^\infty-\tilde{u}_j^\infty$ for $j=1,2$. Using \eqref{uinfinity}, we get $w_1(\hat{x}, k)=-w_2(\hat{x}, k)$ for all $k\in \R$. Hence,  their inverse Fourier transforms must also coincide, i.e., $[\mathcal{F}^{-1}w_1](t)=-
[\mathcal{F}^{-1}w_2](t)$ for all $t\in \R$.
Taking $\delta<\min\{\inf(\hat{x}\cdot D_2)-\Lambda_{\min}, \; \inf(\hat{x}\cdot \tilde{D}_2)-\Lambda_{\min} \}$. Again using Lemma \ref{lem1}, we obtain
\ben
0=[\mathcal{F}^{-1}w_2](t)=[\mathcal{F}^{-1}w_1](t)\quad\mbox{for all}\quad t\in (T_{\min}, T_{\min}+\delta)
\enn
because the interval $(T_{\min}, T_{\min}+\delta)$ lies in the exterior of the supporting intervals of both $\mathcal{F}^{-1} u_2^\infty$ and $\mathcal{F}^{-1} \tilde{u}_2^\infty$.
Combining this with the analyticity of  $[\mathcal{F}^{-1}w_1](t)$ in $t$, we get
\be\label{eq:2}
\quad\quad\quad[\mathcal{F}^{-1}w_1](t)=0\;\;\mbox{for all}\;\;t\in\Big(T_{\min}, \min\left\{\sup(\hat{x}\cdot D_1), \sup(\hat{x}\cdot \tilde{D}_1)\right\}-t_{\min}\Big).
\en
If  $\sup(\hat{x}\cdot D_1)<\sup(\hat{x}\cdot \tilde{D}_1)$, it is seen from Lemma \ref{lem1} that
\be\label{eq:1}
[\mathcal{F}^{-1} u_1^\infty](t^*)=0,\qquad
[\mathcal{F}^{-1} \tilde{u}_1^\infty](t^*)>0
\en
where
\ben
 t^*=\sup(\hat{x}\cdot D_1)-t_{\min}\in I_j=\Big( \inf(\hat{x}\cdot \tilde{D}_1)-t_{\max},  \sup(\hat{x}\cdot \tilde{D}_1)-t_{\min}\Big).
\enn

Obviously, the relations in \eqref{eq:1} contradicts the fact that  $[\mathcal{F}^{-1}w_1](t^*)=0$ by \eqref{eq:2}. This proves  $\sup(\hat{x}\cdot D_1)\geq\sup(\hat{x}\cdot \tilde{D}_1)$. The relation $\sup(\hat{x}\cdot D_1)\leq\sup(\hat{x}\cdot \tilde{D}_1)$ can be proved analogously. Hence, $\sup(\hat{x}\cdot D_1)=\sup(\hat{x}\cdot \tilde{D}_1):=\Lambda_{\max}$ and
\ben
[\mathcal{F}^{-1}w_1](t)=0\quad\mbox{for all}\quad t\in\Big(\Lambda_{\min}-t_{\max}, \Lambda_{\max}-t_{\min}\Big).
\enn
Using again Lemma \ref{lem1} we find that $\mathcal{F}^{-1}u_1^\infty$ and $\mathcal{F}^{-1}\tilde{u}_1^\infty$ also vanish for $t\notin (\Lambda_{\min}-t_{\max}, \Lambda_{\max}-t_{\min})$. Therefore, $w_1\equiv 0$ and $u_1^\infty\equiv \tilde{u}_1^\infty$, which implies $u_2^\infty\equiv \tilde{u}_2^\infty$.
%\enn
\end{proof}

\begin{rem} Once $u_j^\infty(\hat{x}, k)$ ($j=1,2$) can be computed from $u^\infty(\hat{x}, k)$, one can apply the factorization scheme proposed in Sections 2-4 to get an image of the strip $K_{D_j}^{(\hat{x})}$ for $j=1,2$. The numerical implementation of the multi-frequency far-field splitting at a single observation direction is beyond the scope of this paper. We refer to \cite{GHS}  for a numerical scheme splitting far-field patterns over all directions at a fixed frequency.
\end{rem}

\section{Numerical examples}\label{num}

In this section, we present a couple of numerical examples in $\R^2$ \rot{and $\R^3$} to test the proposed factorization method with far-field \rot{and near-field} measurements.  All numerical examples are implemented by MATLAB.

\subsection{\rot{Numerical implements with far-field measurements in $\R^2$}}

We first describe the reconstruction procedure  with  multi-frequency far-field data over a single or multiple observation directions. Unless otherwise stated, we always assume $k_{\min}=0$.  With such a choice  the far-field operator (\ref{def:F}) can be simplified to be
\be\label{def:F1}
(F^{(\hat x)}\phi)(\tau)=\int_0^{k_{\max}} w^{\infty}(\hat x, \tau-s)\phi(s)ds,\quad  L^2(0, k_{\max})\rightarrow L^2(0, k_{\max})
\en
by taking $k_c=0$ and $K =k_{\max}$; see Remark \ref{re3.1}.
In our numerical examples below, we consider $2N-1$ wave-number samples $w^{\infty}(\hat x, k_n), n=1,2,\cdots,N$, and $w^{\infty}(\hat x, -k_n), n=1,2,\cdots,N-1$ of the far field, where
%\ben
$k_n=(n-0.5)\Delta k$, $\Delta k:={k_{\max}}/{N}$.
%\enn
Using the $2N-1$ samples of the far field in (\ref{def:F1}) and applying the midpoint rule, we obtain from (\ref{def:F1}) that
\ben
(F^{(\hat x)}\phi)(\tau_n) \approx \sum_{m=1}^N\, w^{\infty}(\hat x, \tau_n-s_m)\phi(s_m)\Delta k
\enn
where $\tau_n:=n\Delta k$ and $s_m:=(m-0.5)\Delta k$, $n,m=1,2,\cdots,N$. Accordingly, a discrete approximation of $F^{(\hat x)}$ is given by the Toeplitz matrix
\ben \label{matF}
F^{(\hat x)}:=\begin{pmatrix}
w^{\infty}(\hat{x},k_1) &\hspace{-0.1em} \overline{w^{\infty}(\hat{x},k_1)} &\hspace{-0.1em} \cdots &\hspace{-0.1em} \overline{w^{\infty}(\hat{x},k_{N-2})}  &\hspace{-0.1em} \overline{w^{\infty}(\hat{x},k_{N-1})}  \\
w^{\infty}(\hat{x},k_2) &\hspace{-0.1em} w^{\infty}(\hat{x},k_1) &\hspace{-0.1em} \cdots &\hspace{-0.1em} \overline{w^{\infty}(\hat{x},k_{N-3})} &\hspace{-0.1em}\overline{w^{\infty}(\hat{x},k_{N-2})}   \\
\vdots &\hspace{-0.1em} \vdots  &\hspace{-0.1em}  &\hspace{-0.1em}\vdots &\hspace{-0.1em}\vdots \\
w^{\infty}(\hat{x},k_{N-1}) &\hspace{-0.1em} w^{\infty}(\hat{x},k_{N-2}) &\hspace{-0.1em}  \cdots &\hspace{-0.1em} w^{\infty}(\hat{x},k_1)&\hspace{-0.1em} \overline{w^{\infty}(\hat{x},k_1)}\\
w^{\infty}(\hat{x},k_N) &\hspace{-0.1em} w^{\infty}(\hat{x},k_{N-1}) &\hspace{-0.1em}  \cdots &\hspace{-0.1em} w^{\infty}(\hat{x},k_2)&\hspace{-0.1em} w^{\infty}(\hat{x},k_1)\\
\end{pmatrix} \Delta k
\enn
where $\overline{w^{\infty}(\hat{x},k_n)}=w^{\infty}(\hat{x},-k_n)$, $n=1,\cdots,N-1$ and $F^{(\hat x)}$ is a $N\times N$ complex matrix.
For any point $y\in \R^2$ we define the test function vector $\phi_{y}^{(\hat x)}\in \C^N$ from (\ref{test}) by
%Similarly, we discretize the test function $\phi_{y}^{(\hat x)}$ from (\ref{test}) by the test vector
\be \label{testn}
\phi_y^{(\hat{x})}:= \left(\frac{\rot{-}{\rm i} }{ {T}\tau_1} (e^{\rot{{\rm i}  \tau_1 T}}-1)\,e^{-{\rm i} \tau_1 \hat{x}\cdot y}, \;\cdots,\; \frac{\rot{-}{\rm i} }{ {T}\tau_N} (e^{\rot{{\rm i}  \tau_N T}}-1)\,e^{-{\rm i} \tau_N \hat{x}\cdot y}\right),
\en
where $T=t_{max}-t_{min}$.
Denoting by  $\left\{ (  {\tilde \lambda^{(\hat x)}_n}, \psi^{(\hat x)}_n): n=1,2,\cdots,N \right\}$ an eigen-system of the matrix $F^{(\hat x)}$,  \rot{one then} deduces that an eigen-system of the matrix $(F^{(\hat x)})_\#:= |\real (F^{(\hat x)})|+|\ima (F^{(\hat x)})|$ is $\left\{ ( \lambda^{(\hat x)}_n, \psi^{(\hat x)}_n): n=1,2,\cdots,N \right\}$ , where $ \lambda^{(\hat x)}_n:=|\real (\tilde \lambda^{(\hat x)}_n)| +|\ima (\tilde \lambda^{(\hat x)}_n)|$. %We    approximate the indicator function $W^{(\hat x)}$ in (\ref{indicator}) by
%\ben
%W^{(\hat{x})}(y):=\left[\sum_{n=1}^N\frac{\left| \phi^{(\hat{x})}_{y}\cdot \overline{\psi_n^{(\hat{x})} }\right|^2}{ |\lambda_n^{(\hat{x})}|}\right]^{-1}, \quad y\in \R^2,
%\enn
%where $\cdot$ denotes the inner product in $\R^N$.
%Accordingly,
A plot of $W^{(\hat{x})}(y)$ should yield a visualization of the strip $K_D^{(\hat{x})}$ (see \eqref{K} ) containing the source support. In the following numerical examples, the frequency band is taken as $(0,16\pi/6)$ with $k_{\max}=16\pi/6$, $N=16$ and $\Delta k=\pi/6$.  The wave-number-dependent source term $f(x, k)$ is supposed to be given by (\ref{fxt}). We always take $t_{\min}=0$ and $t_{\max}=T=0.1$ unless otherwise specified.

\subsubsection{One observation direction}

 We first consider reconstruction of the strip $K_D^{(\hat{x})}$ from the multi-frequency far-field data $w^\infty(\hat{x}, \pm k_n)$.  In Figure \ref{fig:1dir}, we show a visualization of  reconstructions of three sources supported on a kite-shaped domain at the observation direction $\hat{x}=(\cos\theta, \sin\theta)$ with the angle $\theta\in (0,2\pi]$. The time-dependent source functions are chosen as $S(x, t)$ satisfying the \rot{positivity} assumption \eqref{F}. We choose $S(x, t)=3 {(t+1)}$ and $\theta=\pi/4$ in Figure \ref{fig:1dir} (a); $S(x, t)=3x_1 {(t+1)}$ and $\theta=\pi/2$ in Figure \ref{fig:1dir} (b);  $S(x, t)=3(x_1^2+x_2^2- {4)(t+1)}$ and $\theta=3\pi/4$ in Figure \ref{fig:1dir} (c).  The boundary of  $D$ is also shown in the picture (pink-solid line). As predicted by our theoretical results, the reconstructions nicely approximate the smallest  strip $K_D^{(\hat{x})}$ perpendicular to the observation directions that contains the support. % ($T=0.1,N=16, k=\pi/6$)

\begin{figure}%[tbhp]
\centering
\subfigure[$S=3 {(t+1)}, \theta=\pi/4$]{
\includegraphics[scale=0.18]{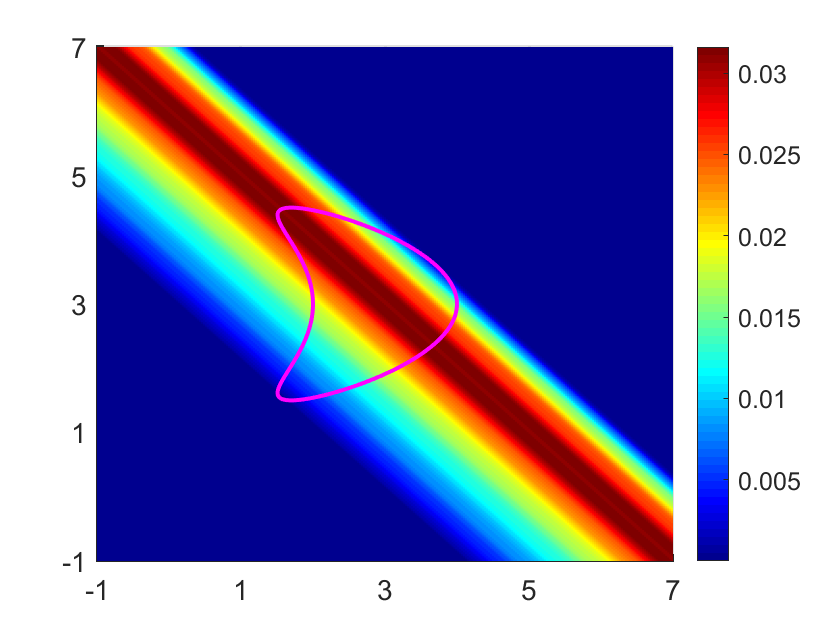}
}
\subfigure[$S=3x_1 {(t+1)}, \theta=\pi/2$ ]{
\includegraphics[scale=0.18]{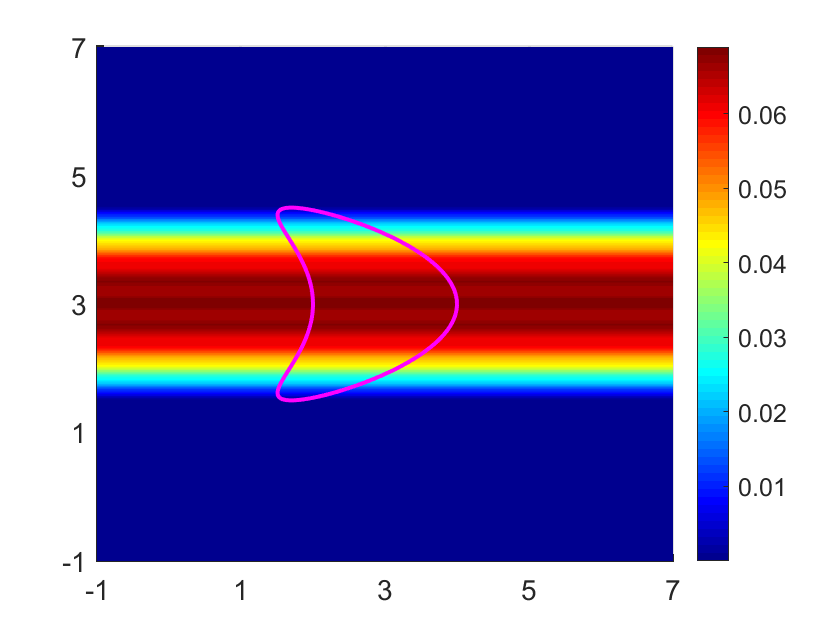}
}
\subfigure[$S=3(x_1^2+x_2^2- {4)(t+1)}, \theta=3\pi/4$]{
\includegraphics[scale=0.18]{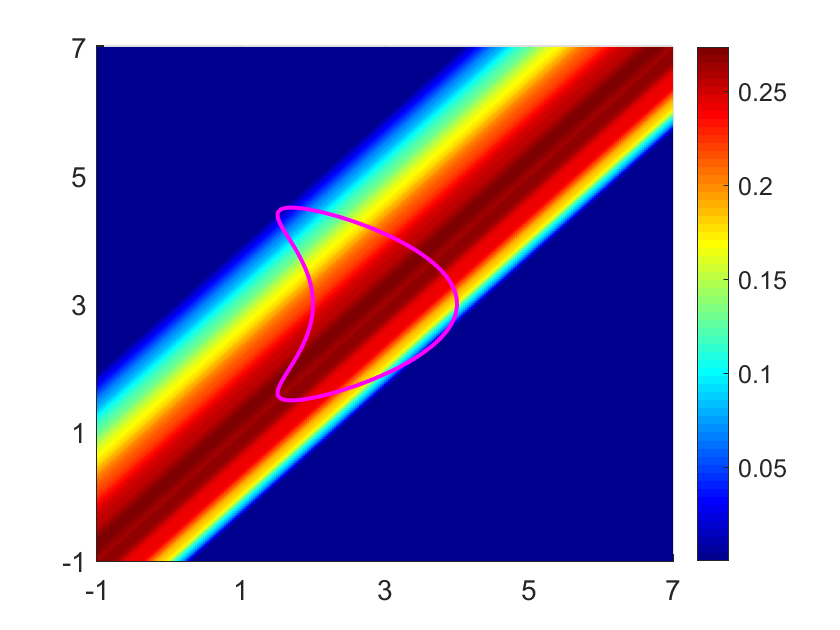}
}
\caption{Reconstructions using a single observation direction and multi-frequency far-field data for a kite-shaped  support. We choose $t_{\min}=0$ and $t_{\max}=T=0.1$.
} \label{fig:1dir}
\end{figure}

\begin{figure}[H]
\centering
\subfigure[$T=1$]{
\includegraphics[scale=0.18]{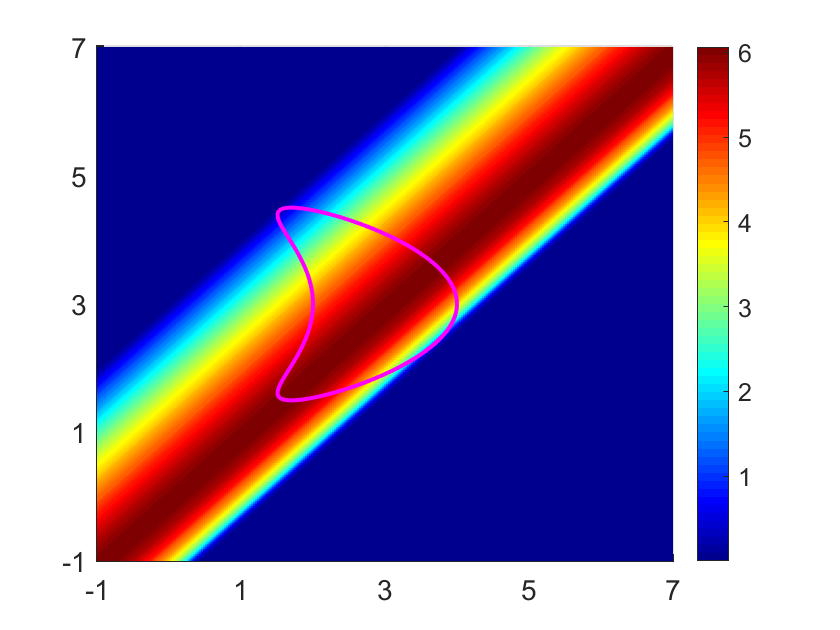}
}
\subfigure[$T=2$ ]{
\includegraphics[scale=0.18]{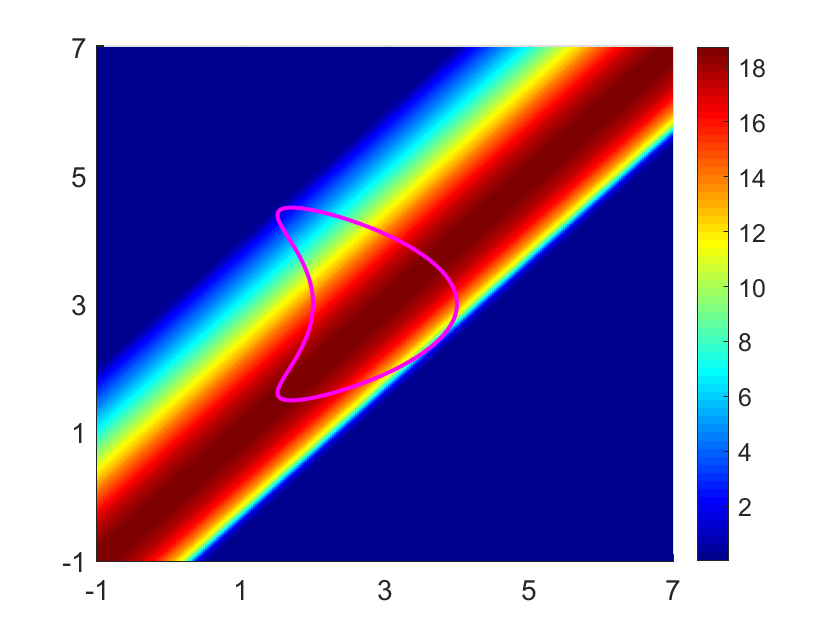}
}
\subfigure[$T=3$]{
\includegraphics[scale=0.18]{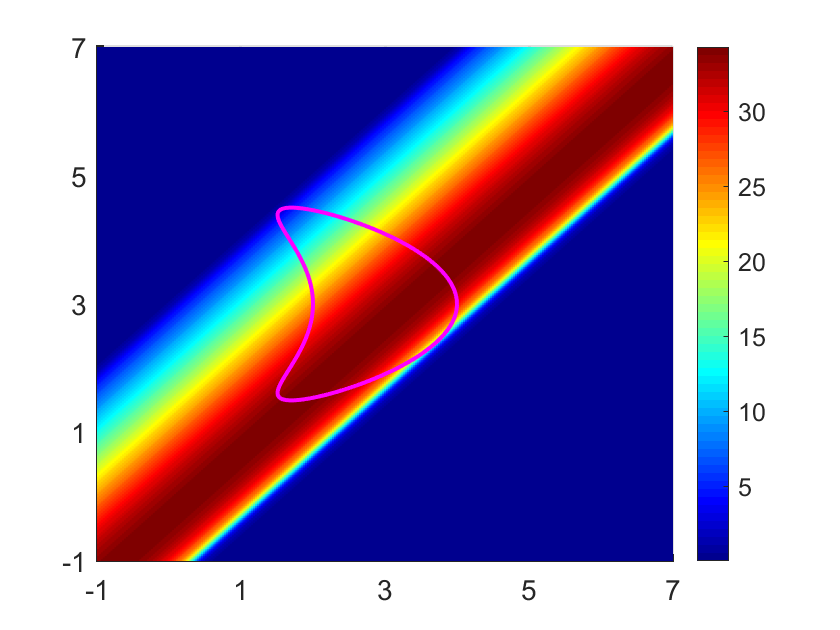}
}
\caption{ Reconstructions of a kite-shaped support with $S=3(x_1^2+x_2^2- {4)(t+1)}$ and $\theta=3\pi/4$ with different \rot{inverse} Fourier transform windows $(0, T)$.} \label{fig:0Tj}
\end{figure}
Next, we continue the numerical example for $S(x,t)=3(x_1^2+x_2^2- {4)(t+1)}$ and $\theta=3\pi/4$ in Figure \ref{fig:1dir} (c) but with different \rot{inverse} Fourier transform windows $(t_{\min}, t_{\max})$.   We take $t_{\min}=0, t_{\max} =T$ with $T=1,2,3$ in Figure \ref{fig:0Tj}. The \rot{inverse} Fourier transform window is taken to be $(t_{\min}, t_{\max})=  (T_0, T_0+T)$ in Figure \ref{fig:Tj}, with $T=0.1$ and $T_0=1, 2, 3$ . In the last case, one needs to discretize the  test vector (\ref{testn}). %is discretized as
%\ben
%\phi_y^{(\hat{x})}:= \left(\frac{{\rm i}}{ {T} \tau_1} (e^{-{\rm i}  \tau_1 (T_0+T)}-e^{-{\rm i}  \tau_1 T_0})\,e^{-{\rm i} \tau_1 \hat{x}\cdot y}, \;\cdots,\; \frac{{\rm i} }{ {T} \tau_n} (e^{-{\rm i}  \tau_n (T_0+T)}-e^{-{\rm i}  \tau_n T_0})\,e^{-{\rm i} \tau_n \hat{x}\cdot y}\right).% \in \C^N.
%\enn
It is  clearly shown that the source still lies in the smallest strip perpendicular to the observation direction. The numerical examples  in Figures \ref{fig:0Tj} and \ref{fig:Tj} show that \rot{our inversion algorithm is feasible for any $t_{\min}$ and $t_{\max}$.}

\begin{figure}[H]
\centering
\subfigure[$T_0=1, T=0.1$]{
\includegraphics[scale=0.18]{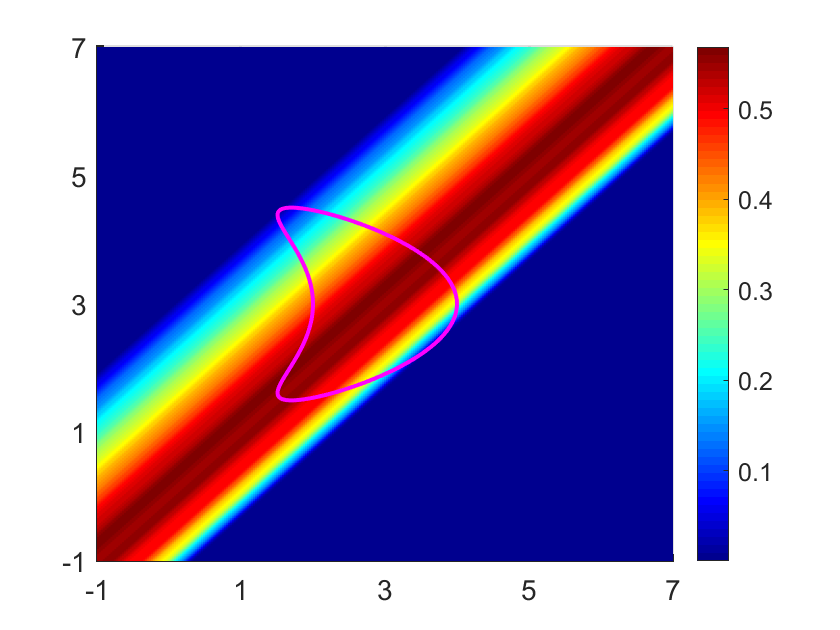}
}
\subfigure[$T_0=2, T=0.1$ ]{
\includegraphics[scale=0.18]{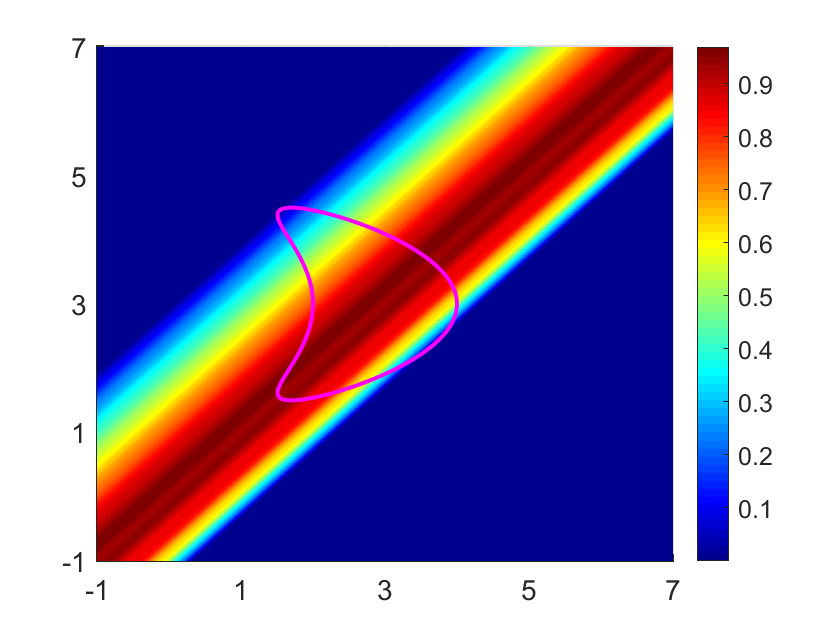}
}
\subfigure[$T_0=3, T=0.1$]{
\includegraphics[scale=0.18]{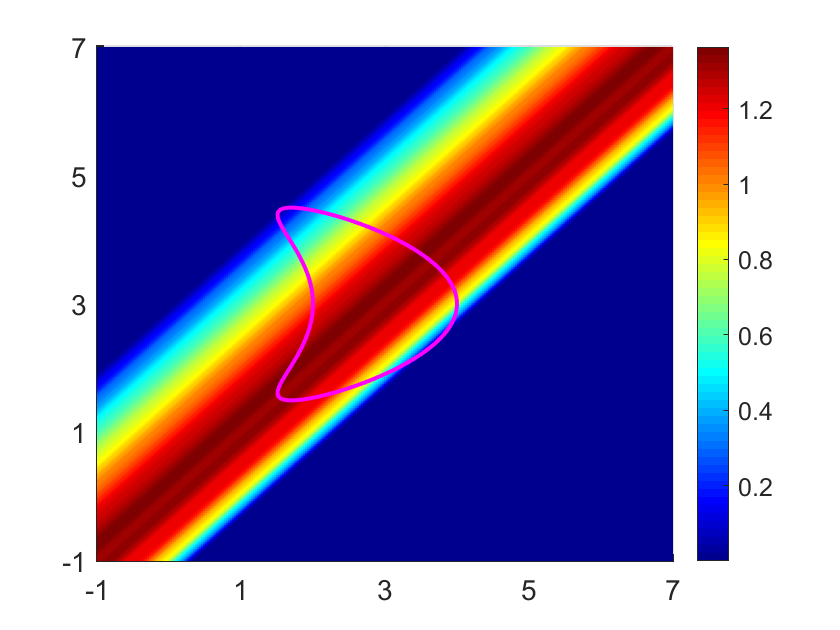}
}
\caption{Reconstructions of a kite-shaped support with $S=3(x_1^2+x_2^2- {4)(t+1)}$, $\theta=3\pi/4$ and with different \rot{inverse} Fourier transform windows of the form $(T_0, T_0+T)$.} \label{fig:Tj}
\end{figure}

\subsubsection{Multiple observation directions}

We present the reconstruction of a kite-shaped source using $M$ observation directions with the observation angles $\theta_m=\frac{m-1}{M}\pi$, $m=1,\cdots,M$. %Here, one should change the indicator function into
%\ben
%W^{(\hat{x})}(y):=\left[\sum _{m=1}^M\sum_{n=1}^N\frac{\left| \phi^{(\hat{x}_{m})}_{y}\cdot \overline{\psi_n^{(\hat{x}_{m})} }\right|^2}{ |\lambda_n^{(\hat{x}_{m})}|}\right]^{-1}, \quad y\in \R^2.
%\enn
%where the test function $\phi_y^{(\hat{x}_m)}$ is given
%by (\ref{testn}), and $\left\{(\lambda_n^{(\hat{x}_{m})} , \psi_n^{(\hat{x}_{m})}): n=1,\cdots, N\right\} $ denotes an eigensystem of the operator  {$(F^{(\hat{x}_m)})_\#$}.

We show in Figure \ref{fig:mdir} a visualization of the reconstructed source with $S(x, t)=3(x_1^2+x_2^2) {(t+1)}$. Since the observation directions are perpendicular to each other if $M=2$, the strips $K_D^{(\hat{x}_1)}$ and $K_D^{(\hat{x}_2)}$ are also perpendicular to each other as shown in Figure \ref{fig:mdir} (a). It is clear that intersection of the strips contains the source support in Figure \ref{fig:mdir} (a), (b) and (c) , which approximates the convex hull of the support. Of course the number of observation directions affects reconstruction qualities: the more the directions, the better the reconstructions.

\begin{figure}[H]
\centering
\subfigure[$M=2$]{
\includegraphics[scale=0.18]{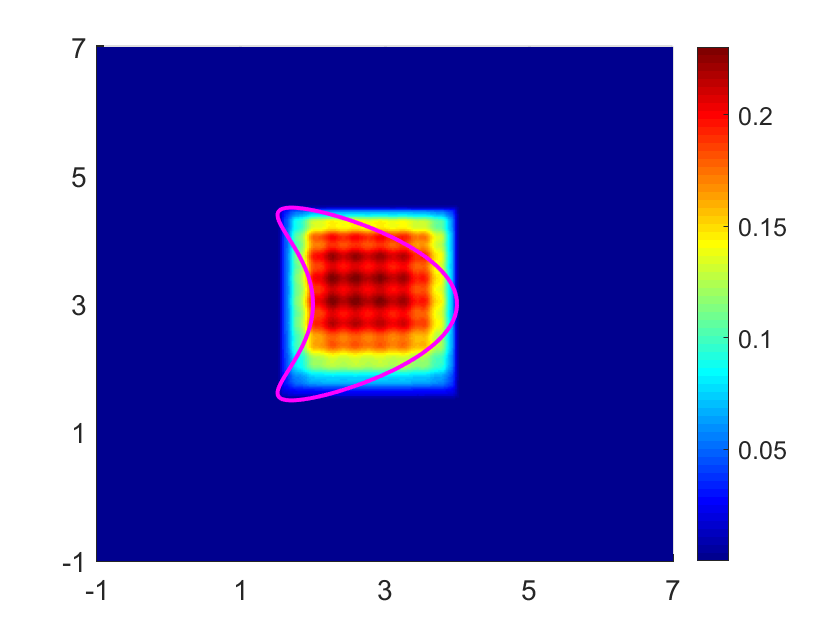}
}
\subfigure[$M=5$ ]{
\includegraphics[scale=0.18]{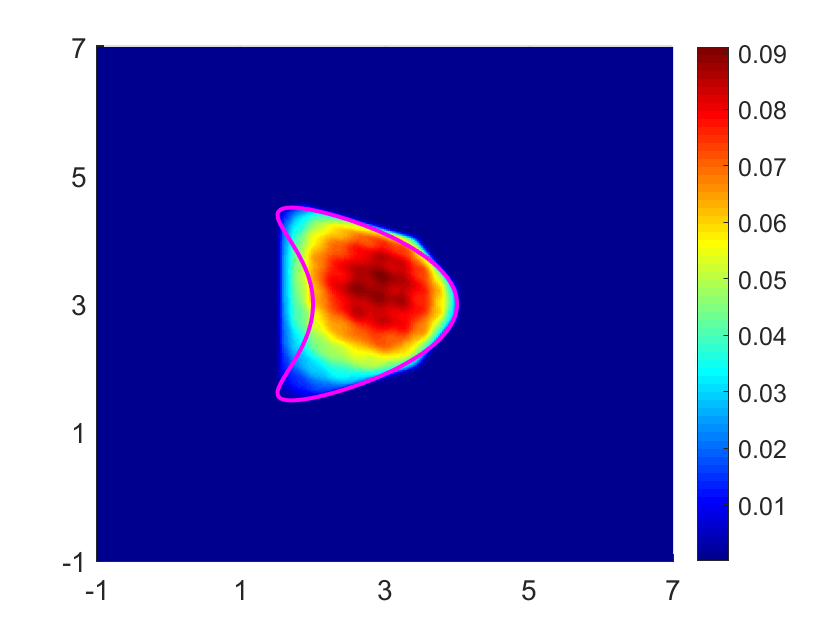}
}
\subfigure[$M=8$]{
\includegraphics[scale=0.18]{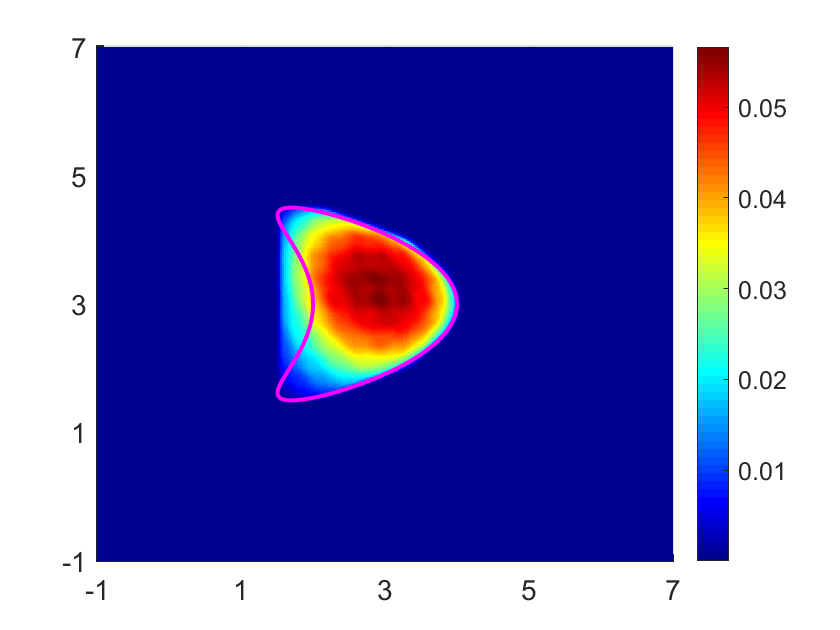}
}
\caption{Reconstructions of a kite-shaped  source for $S=3(x_1^2+x_2^2) {(t+1)}$ with $M$ observation directions. The \rot{inverse} Fourier transform window $(t_{\min}, t_{\max})$ is chosen  as  $t_{\min}=0$ and $t_{\max}=0.1$. } \label{fig:mdir}
\end{figure}

%\subsection{The sensitivity of $T$}

%To get an idea about the sensitivity of the algorithm with respect to $T$ in the source, we take different $T$ for testing.
We continue the numerical example in Figure \ref{fig:mdir} (c) by choosing different \rot{inverse} Fourier transform time windows $(0,T)$. The resulting reconstructions are shown in Figure \ref{fig:T} with three different choices $T=1, 5, 7$. The results are getting worse with increasing $T$, but they still contain useful information on the location and shape of the source \rot{support $D$}. In other numerical examples, we all take $T=0.1$ to get precise reconstructions.

\begin{figure}%[H]
\centering
\subfigure[$T=1$]{
\includegraphics[scale=0.18]{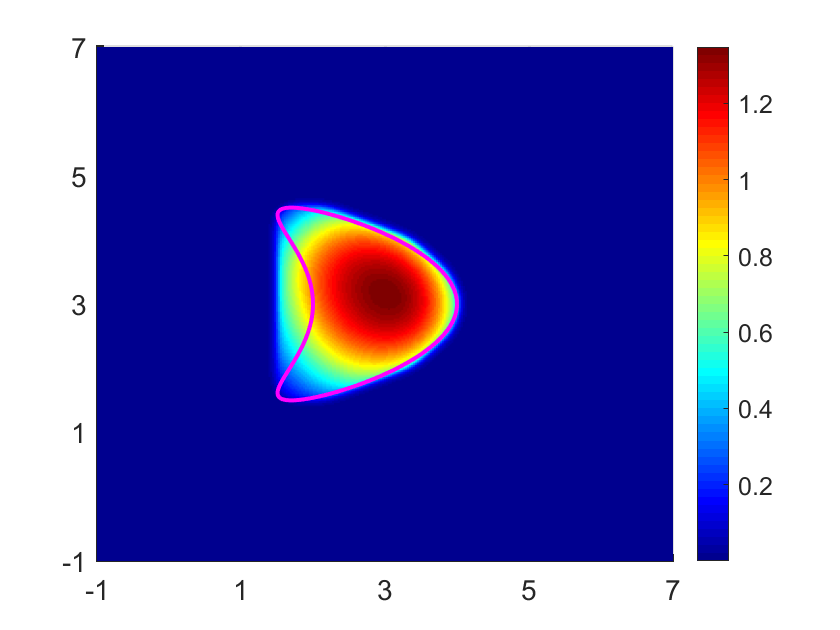}
}
\subfigure[$T=5$ ]{
\includegraphics[scale=0.18]{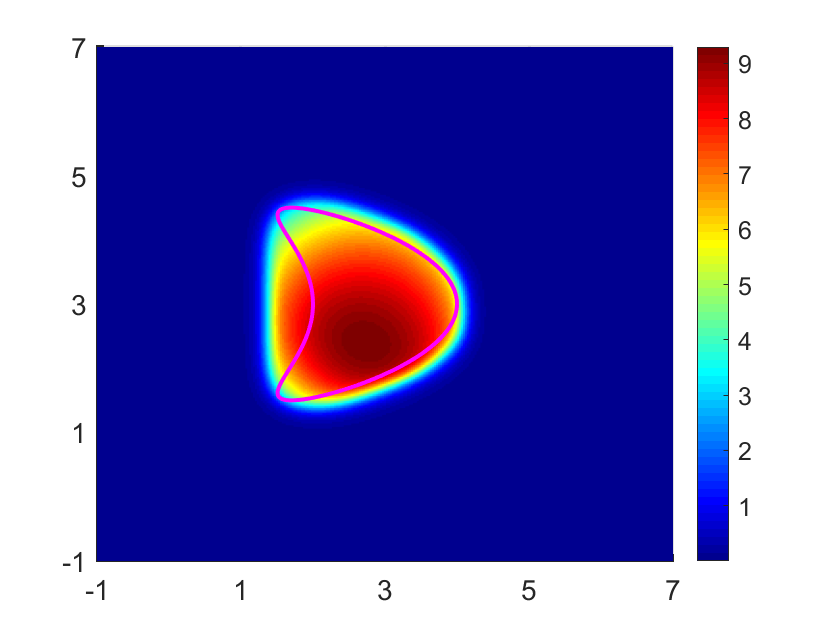}
}
\subfigure[$T=7$]{
\includegraphics[scale=0.18]{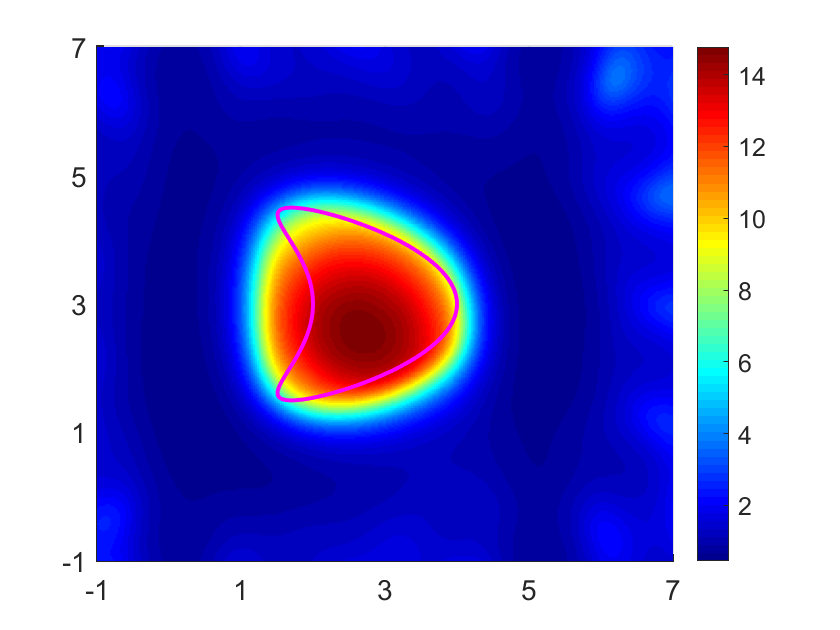}
}
\caption{Reconstructions from $M=8$ observation directions for a kite-shaped support. We choose $S=3(x_1^2+x_2^2) {(t+1)}$ and different \rot{inverse} Fourier transform windows of the form $(0, T)$.} \label{fig:T}
\end{figure}

Next
we consider a source with two disconnected components: one component is kite-shaped and the other one is elliptic. We choose different source functions in Figure \ref{fig:tdir}.
It is shown that the two components are both precisely recovered using $8$ observation directions. It is worth mentioning that the \rot{inverse} Fourier transform time widow $(0, T)$ should not be too big in this case. If $T$ is increasing from 1 to 5, the images will be distorted; see Figure \ref{fig:Tdir}. This is due to the reason that the wave-fields radiated from the two components and received by the sensors cannot be split.  The Figure \ref{fig:Tdir} (b) can be improved if we increase the number of frequencies. However, if $T$ is larger than the distance between the two components (see Figure \ref{fig:Tdir} (c)), the two components cannot be well separated. Instead, the convex hull of the union of these two components  can be recovered.
\begin{figure}%[H]
\centering
\subfigure[$S=3(x_1+3) {(t+1)}$]{
\includegraphics[scale=0.18]{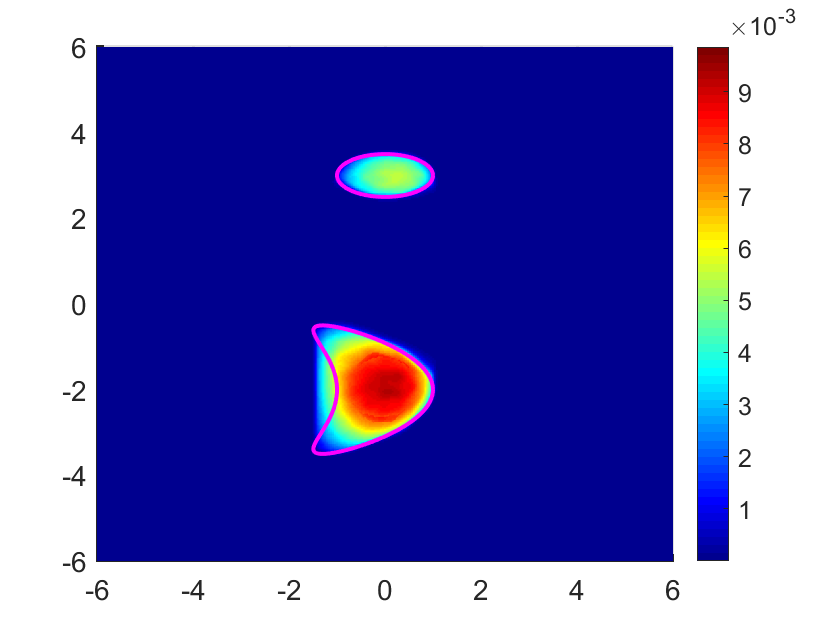}
}
\subfigure[$S=3(x_2+3) {(t+1)}$ ]{
\includegraphics[scale=0.18]{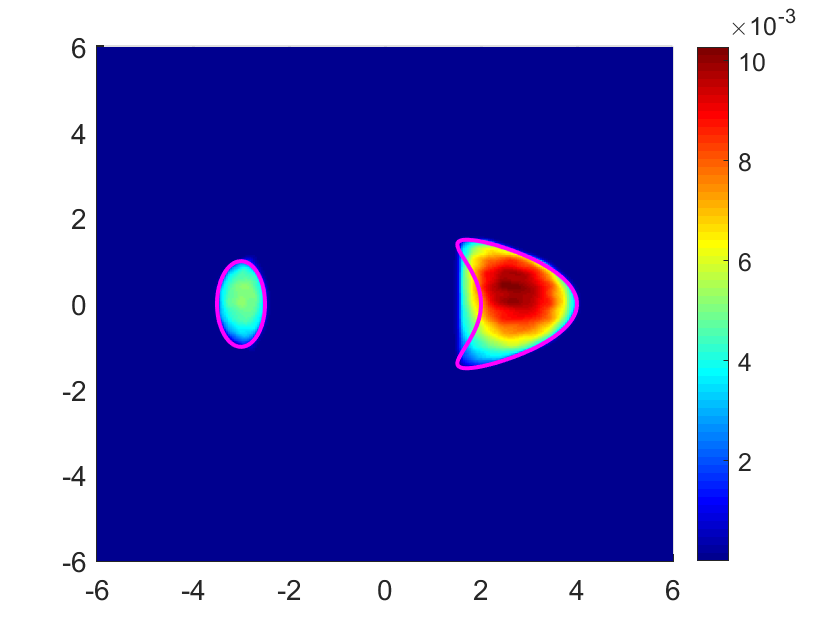}
}
\subfigure[$S=3(x_1^2+x_2^2) {(t+1)}$]{
\includegraphics[scale=0.18]{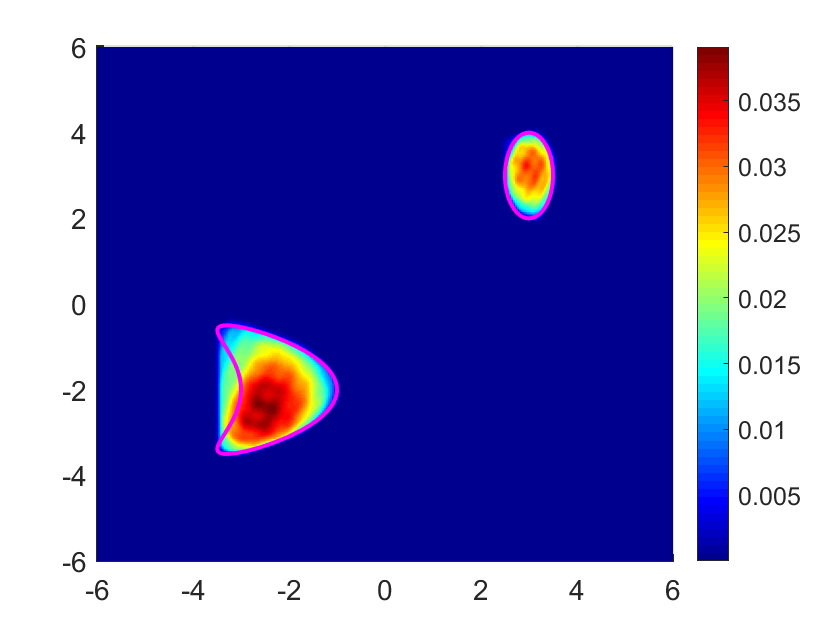}
}
\caption{Reconstructions of the shape of a source with two components from $8$  observation directions. The \rot{inverse} Fourier time window is $(0, T)$ with $T=0.1$. } \label{fig:tdir}
\end{figure}

\begin{figure}%[H]
\centering
\subfigure[$T=1$]{
\includegraphics[scale=0.18]{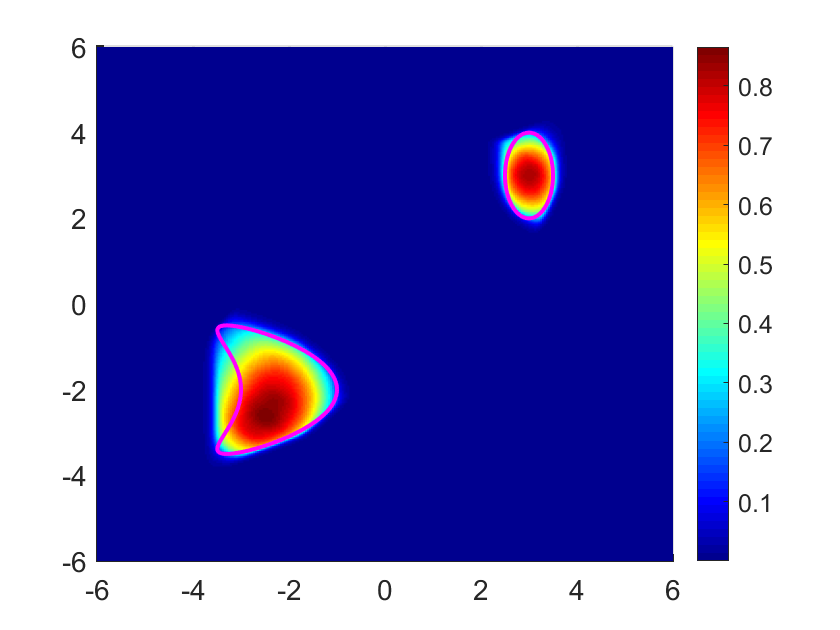}
}
\subfigure[$T=3$ ]{
\includegraphics[scale=0.18]{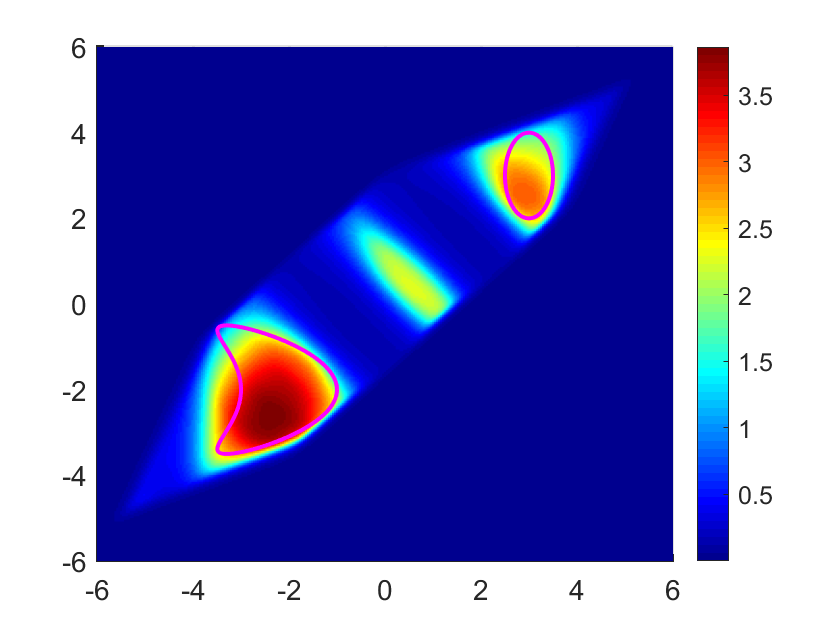}
}
\subfigure[$T=5$]{
\includegraphics[scale=0.18]{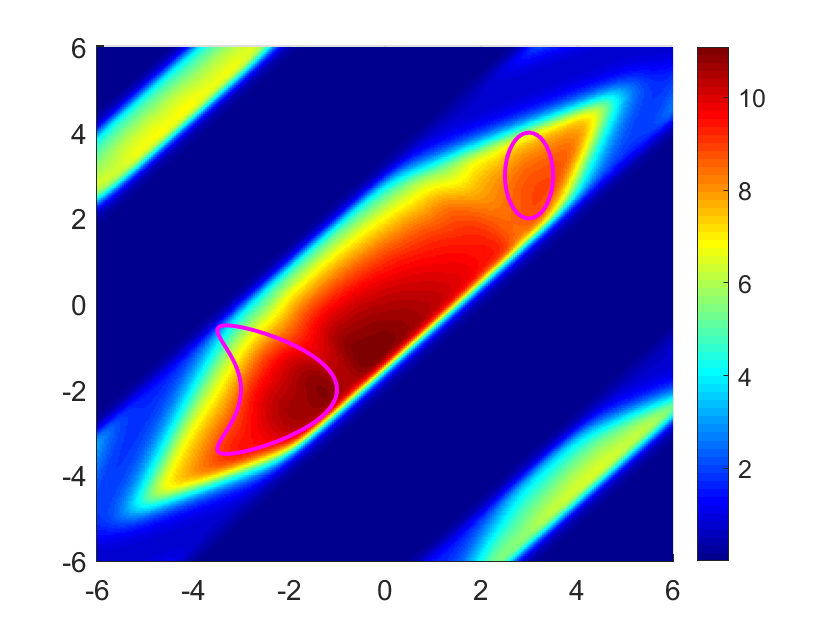}
}
\caption{Reconstructions of two disjoint components from $8$  observation directions with different  time windows $(0, T)$. The source function is $S=3(x_1^2+x_2^2) {(t+1)}$.} \label{fig:Tdir}
\end{figure}

%\subsection{The sensitivity of noise}

\begin{figure}[H]
\centering
\subfigure[$\delta=0$]{
\includegraphics[scale=0.19]{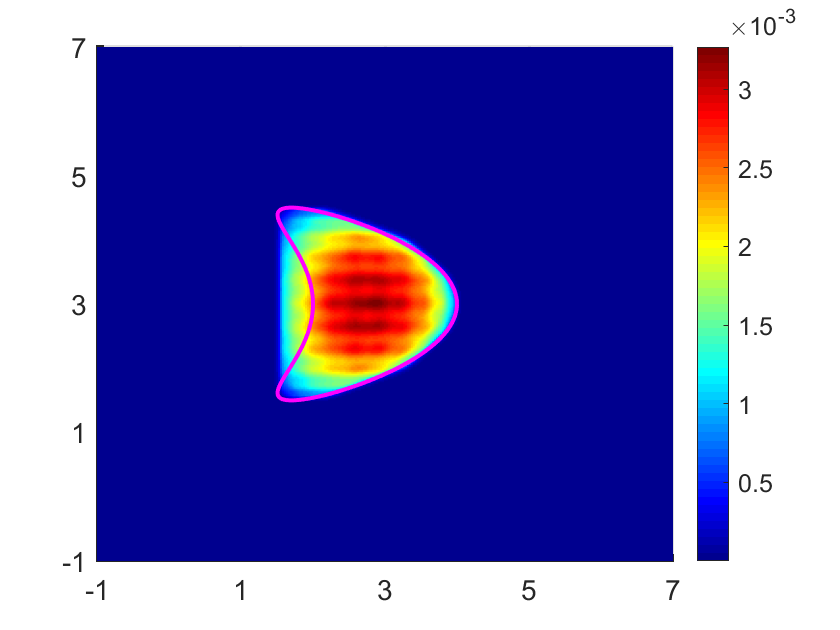}
}
\subfigure[$\delta=2\% $ ]{
\includegraphics[scale=0.19]{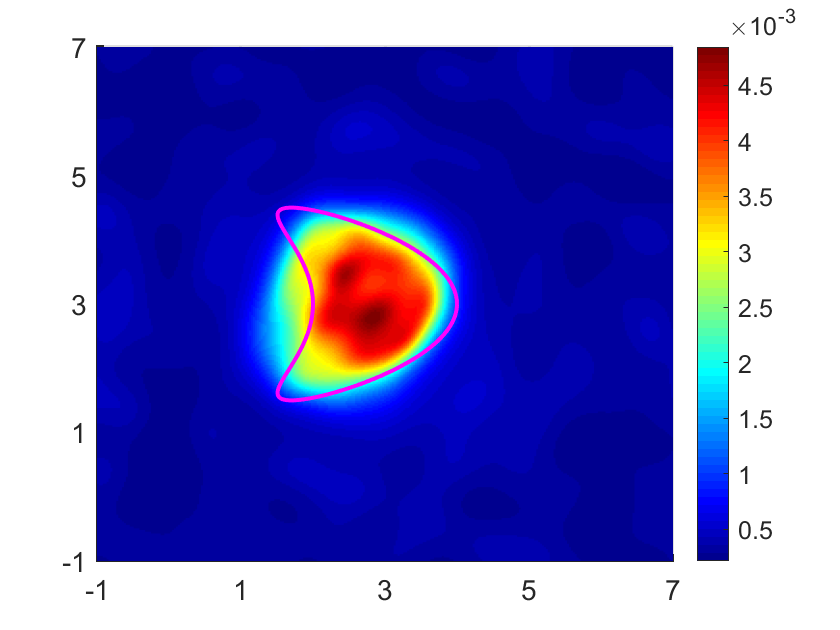}
}

\subfigure[$\delta=5\% $ ]{
\includegraphics[scale=0.19]{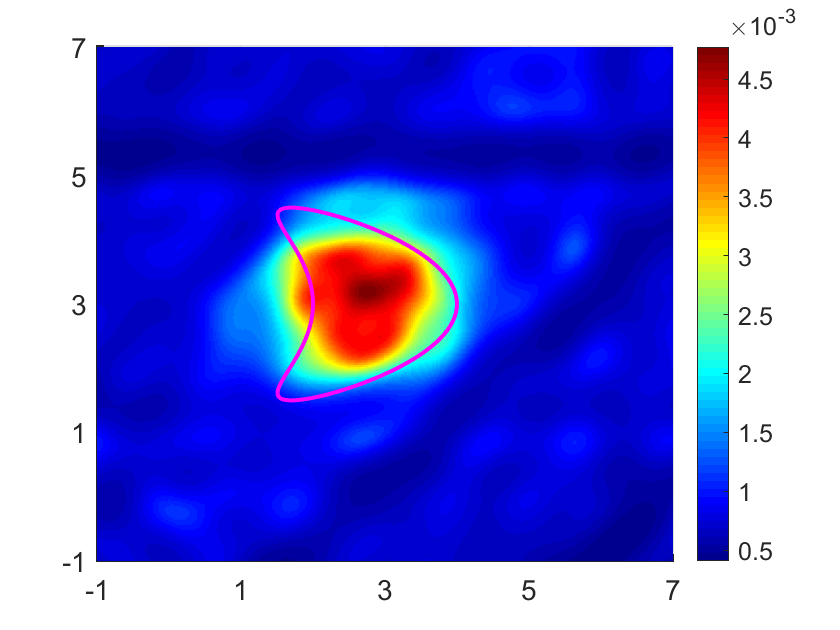}
}
\subfigure[$\delta=10\% $]{
\includegraphics[scale=0.19]{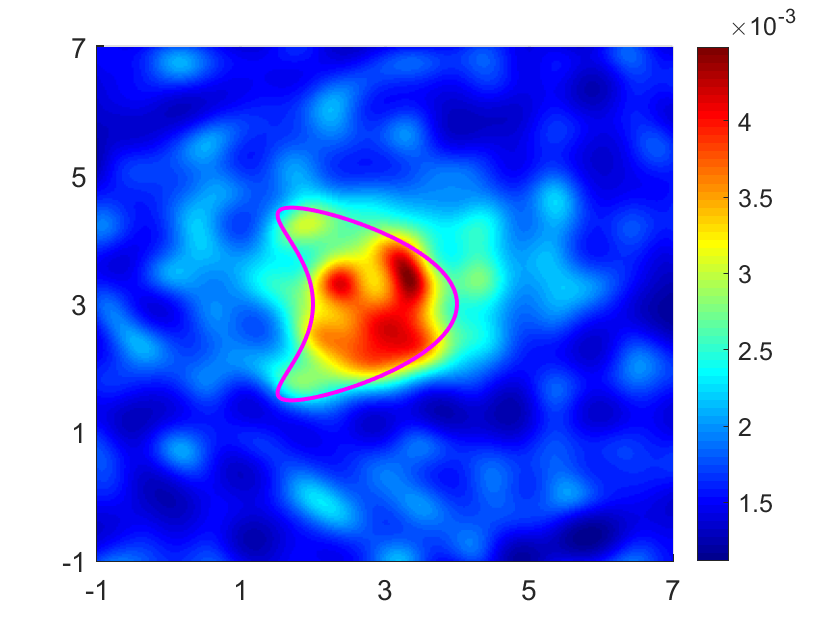}
}
\caption{Reconstructions  of a kite-shaped  source from $8$  observation directions with $S=3 {(t+1)}$ at different noise levels $\delta$.} \label{fig:noise}
\end{figure}

To test the sensitivity of the algorithm with respect to measurement noise, we pollute the far-field data matrix  by
%\ben
$F^{(\hat x)}_\delta:=F^{(\hat x)}+\delta\, \|F^{(\hat x)}\|_2\,\mathcal{M}$,
%\enn
where $\delta$ is the noise level and $\mathcal{M}\in \R^{N\times N}$ is a uniformly distributed random -matrix  with the random variable ranging from $-1$ to $1$. \rot{We present the reconstruction in the noise-free case in  Figure \ref{fig:noise} (a). }The resulting reconstructions are shown in Figure \ref{fig:noise} \rot{ (b), (c) and (d)} at three noise levels. The images are clearly getting distorted at higher noise levels, but the location of the source can still be well-captured.

\subsection{Numerical implements with near-field measurements in $\R^3$} \label{near-3d}
%In this subsection, we present numerical reconstructions of the source support $D$ from the multi-frequency near-field measurements $\{w(x_j,k):x_j \in \pa B_R, k\in(k_{\min}, k_{\max}), j=1,2,...,     M\}$ at sparse observation points. In  view of Corollary \ref{indicator-near}, the smallest annular domain containing the source support  and centered at $x$ can be reconstructed with the indicator function $\widetilde{W}^{(x)}(y)$ in (\ref{near-indicator}). The source support $D$ can be  imaged by plotting the truncated indicator
% \be \label{near-indicator-m}
%\widetilde{W}(y):=\left[\sum_{j=1}^M \sum_{n=1}^N \frac{| \tilde{\phi}^{(x_j)}_{y} \cdot \tilde{\psi}_n^{(x_j)} |^2}{ |\tilde \lambda_n^{(x_j)}|}\right]^{-1}, \qquad y\in \R^3.
%\en
In the following numerical examples, the frequency band is also taken as $(0,16\pi/6)$ with $k_{\max}= 16\pi/6, N = 16$ and $\Delta k = \pi/6$.  \rot{Here we clarify that an iso-surface represents the points in three-dimensional space where the indicator function  $\widetilde W(y)$ or $ W(y)$ has a constant value and that an iso-surface level typically refers to a specific value of the indicator function   $\widetilde W(y)$  or $W(y)$ in space.}
In the first example, we use the indicator function $\widetilde{W}^{(x)}(y)$ of (\ref{near-indicator}) to reconstruct the annulus $\widetilde K_D^{(x)}$ \rot{in (\ref{K})} for a cube.
We take the temporal and spatial dependent source functions to be $F(x,t)=(x_1^2+x_2^2+x_3^2+1)(t+1)$ and the support of the source is assumed to be $\rot{\overline D}\times\rot{[}t_{\min}, t_{\max}\rot{]}$. The cube $D$ is  defined by
%\ben
$D=\{(x_1, x_2, x_3): |x_1|\leq0.5, |x_2|\leq0.5, |x_3|\leq0.5\}$ (see Figure \ref{fig:near-1} (a)).
%\enn
We take $t_{\min}=0$, $t_{\max}=0.1$ and set the measurement point at $(1.5, 0 ,0)$. Then $\widetilde{W}^{(x)}(y)$ is plotted over the searching domain $[-1.5, 1.5]^3$ in Figure \ref{fig:near-1} (b) and (c). We present a slice of the reconstruction at $y_2=0$ in Figure \ref{fig:near-1} (b), from which we conclude that the cross of the plane $y_2=0$ with the smallest annulus containing the square (in pink) and centered at $x=(1.5,0,0)$ is nicely reconstructed. Figure \ref{fig:near-1} (c) illustrates an iso-surface of the reconstruction at the iso-level $\rot{1}\times10^{-3}$. The iso-surfaces perfectly enclose the cube-shaped support.
\begin{figure}%[H]
\centering
\subfigure[Geometry of a cubic support]{
\includegraphics[scale=0.18]{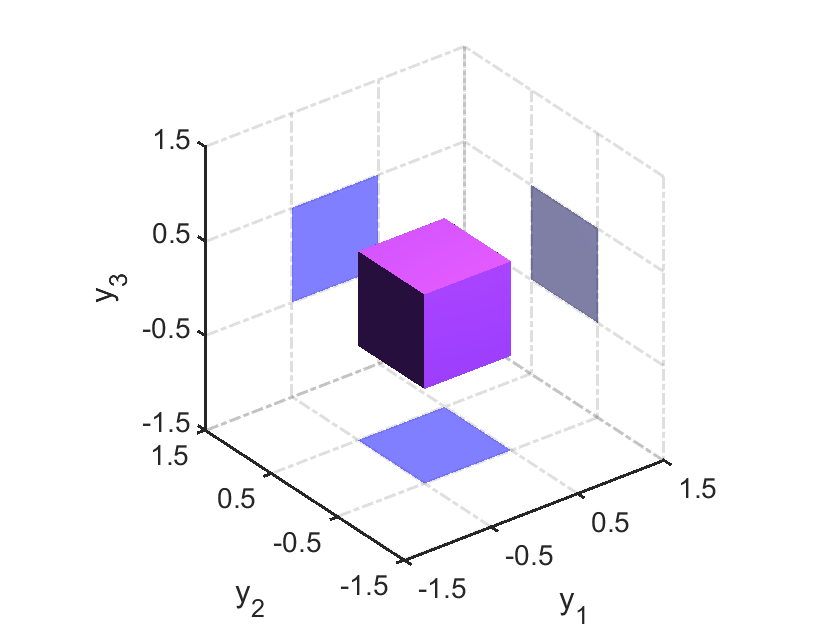}
}
\subfigure[A slice  at $y_2=0$ ]{
\includegraphics[scale=0.18]{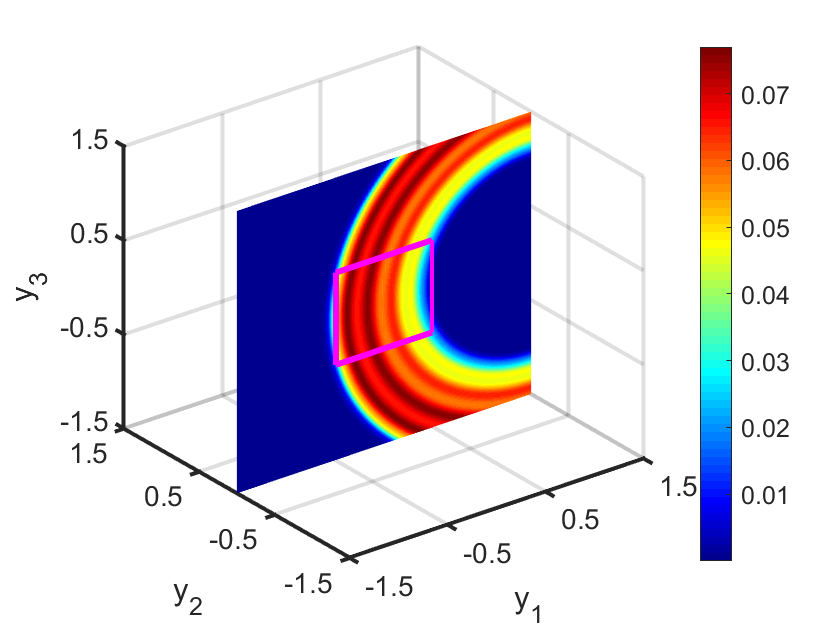}
}
\subfigure[Iso-surface level $=\rot{1}\times10^{-3}$]{
\includegraphics[scale=0.18]{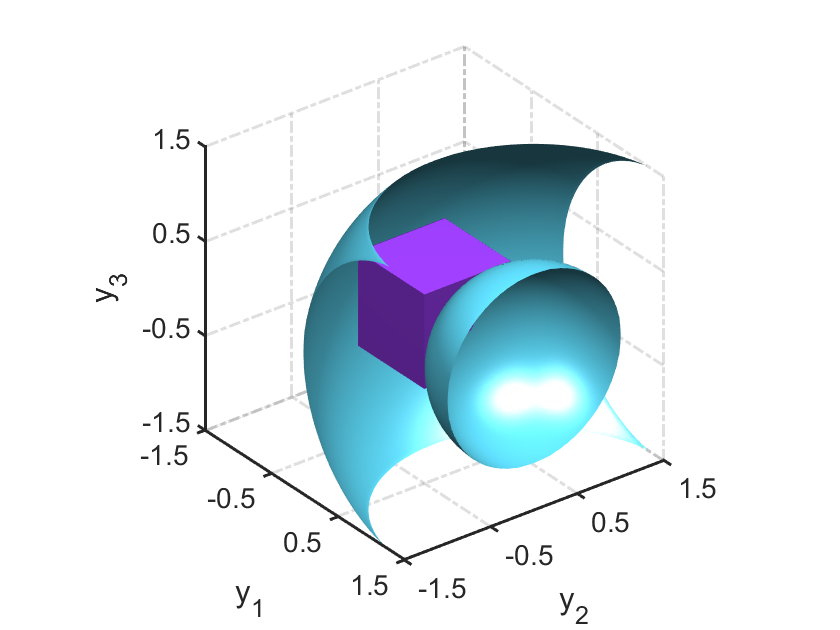}
}
\caption{Reconstructions of a cube from the data measured at one observation point $(1.5, 0, 0)$.% The Fourier transform window is $(t_{\min}, t_{\max})=(0,0.1)$.
}\label{fig:near-1}
\end{figure}
Next we continue the above test with multiple observation points. A visualization of the indicator function %$\widetilde{W}(y)$ in (\ref{near-indicator-m})
 is shown in Figure \ref{fig:near-2} with six observation points $\{ (1.5,0,0), (-1.5,0,0), (0,1.5,0), (0,-1.5,0), (0,0,1.5), (0,0,-1.5)\}$. Figure \ref{fig:near-2} (a) presents an iso-surface of the reconstruction at  the iso-level $\rot{5\times10^{-3}}$ and the projections of the images onto the $oy_1 y_2$, $oy_1 y_3$ and $oy_2 y_3$ planes. It is clearly shown that projections are all squares $[-0.5,0.5]^2$,  justifying the accuracy of our 3D reconstructions. Figure \ref{fig:near-2} (b) and (c) illustrate  slices of the reconstructions at the planes $y_2, \,y_3=0$ and $y_1,\,y_2=0$ using the data at six observation points. These slices also confirm the accuracy of the 3D reconstructions.

\begin{figure}%[H]
\centering
\subfigure[Iso-surface level $=\rot{5\times10^{-3}}$]{
\includegraphics[scale=0.18]{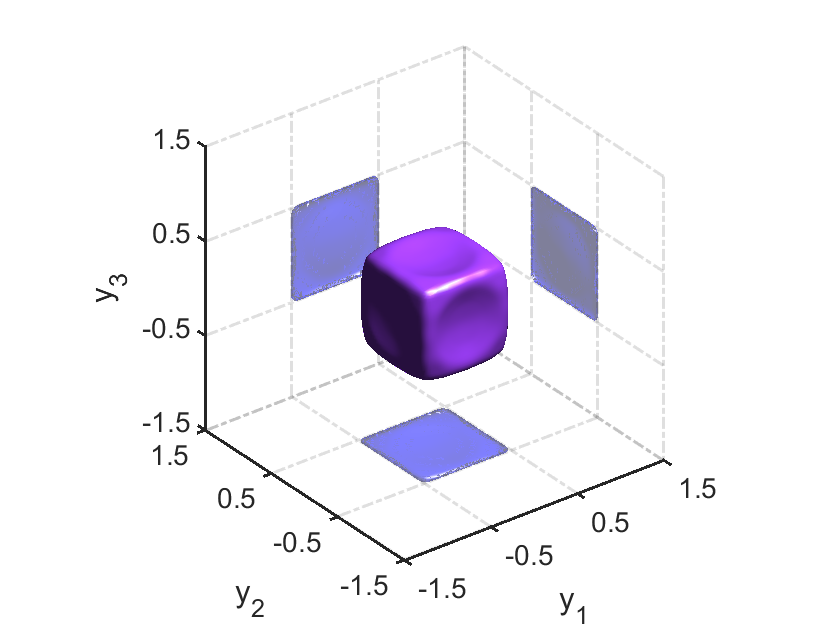}
}
\subfigure[Slices  at $y_2,y_3=0$ ]{
\includegraphics[scale=0.18]{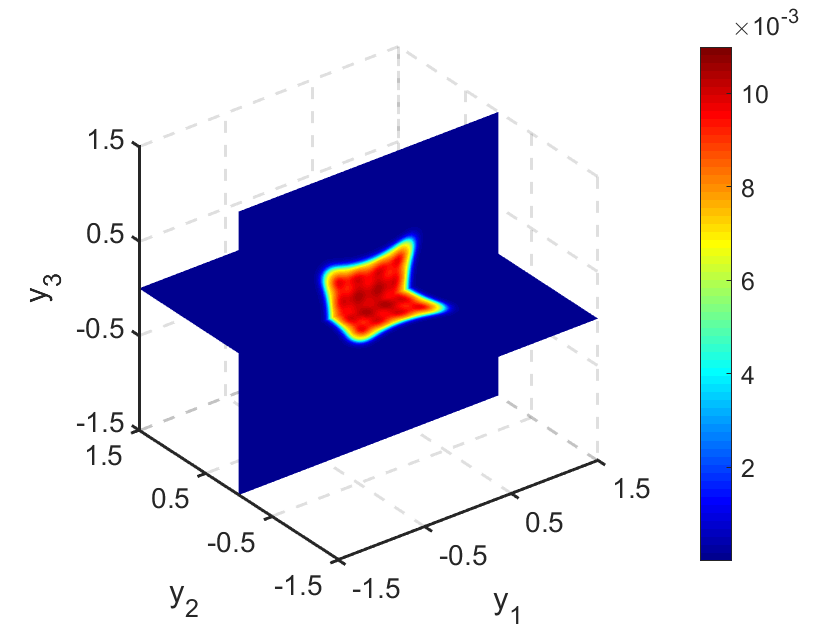}
}
\subfigure[Slices  at $y_1,y_2=0$]{
\includegraphics[scale=0.18]{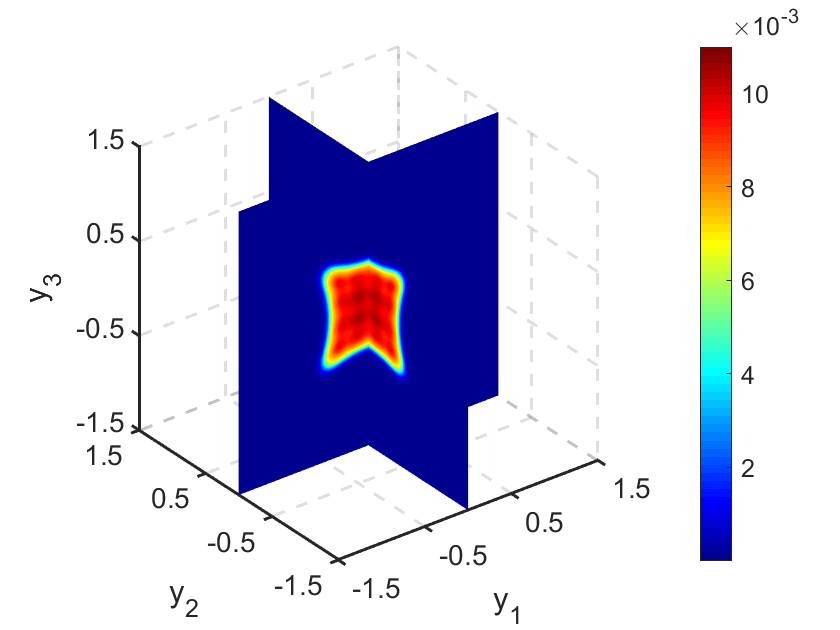}
}
\caption{Reconstructions of a cube from six observation points. %The Fourier transform window is $(0,0.1)$.
 } \label{fig:near-2}
\end{figure}

In Figure \ref{fig:near-3,fig:near-4,fig:near-5}, we show iso-surfaces and slices of the 3D reconstruction of the cubic source with a longer radiating period $(t_{\min}, t_{\max})$. Different $\rot{inverse}$ Fourier transform windows from the data measured at six observation points are used. We choose the radiating period (resp. \rot{inverse} Fourier transform window) as $(0,1)$ in Figure \ref{fig:near-3}, as $(0,3)$ in Figure \ref{fig:near-4} and as $(0,5)$ in Figure \ref{fig:near-5}. It can be observed that, even for a long duration $T=t_{\max}-t_{\min}$, satisfactory inversions for capturing the shape and  location of the source \rot{support $D$} can be achieved by taking different iso-surface levels. We also present some slices of the reconstructions, which all confirm effectiveness of our algorithm.

\begin{figure}%[H]
\centering
\subfigure[Iso-surface level$=\rot{1.3\times10^{-1}}$]{
\includegraphics[scale=0.18]{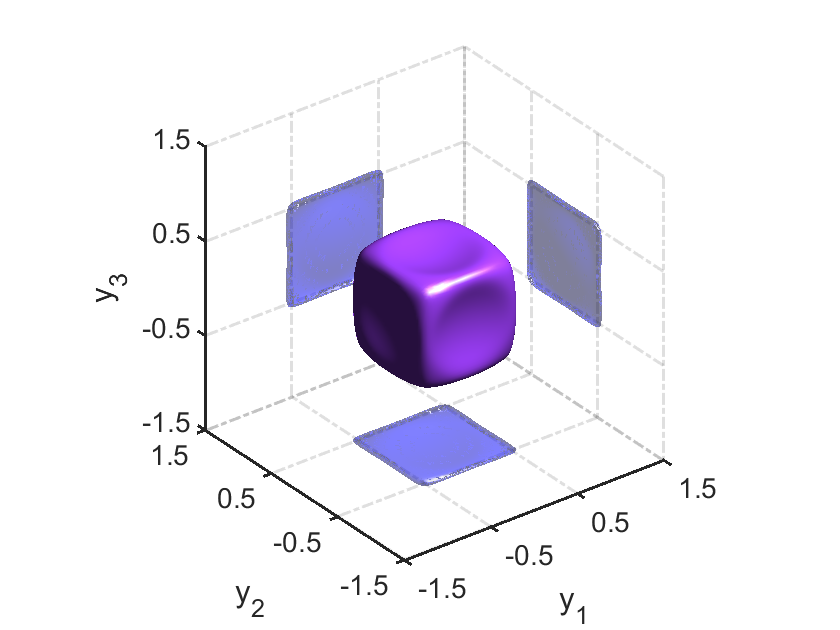}
}
\subfigure[Slices  at $y_2,y_3=0$ ]{
\includegraphics[scale=0.18]{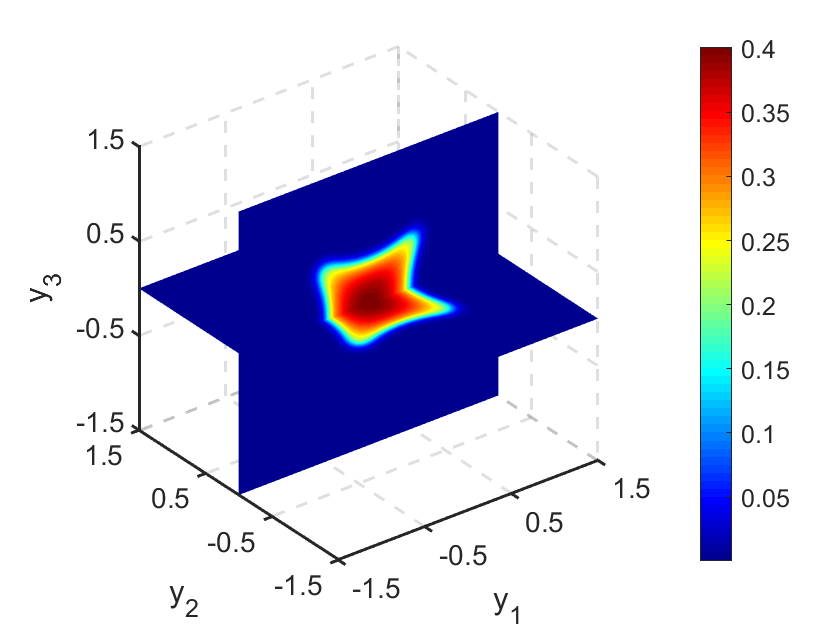}
}
\subfigure[Slices  at $y_1,y_2=0$]{
\includegraphics[scale=0.18]{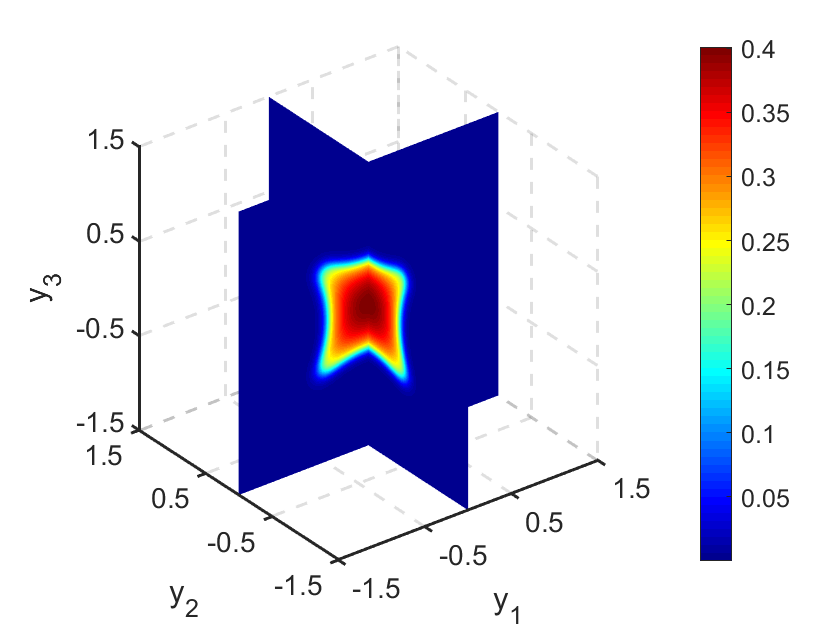}
}
\caption{Reconstructions of a cube from six observation points.
%The Fourier transform window is  $(0,1)$.
}\label{fig:near-3}
\end{figure}

\begin{figure}%[H]
\centering
\subfigure[Iso-surface level $=\rot{9\times 10^{-1}}$]{
\includegraphics[scale=0.18]{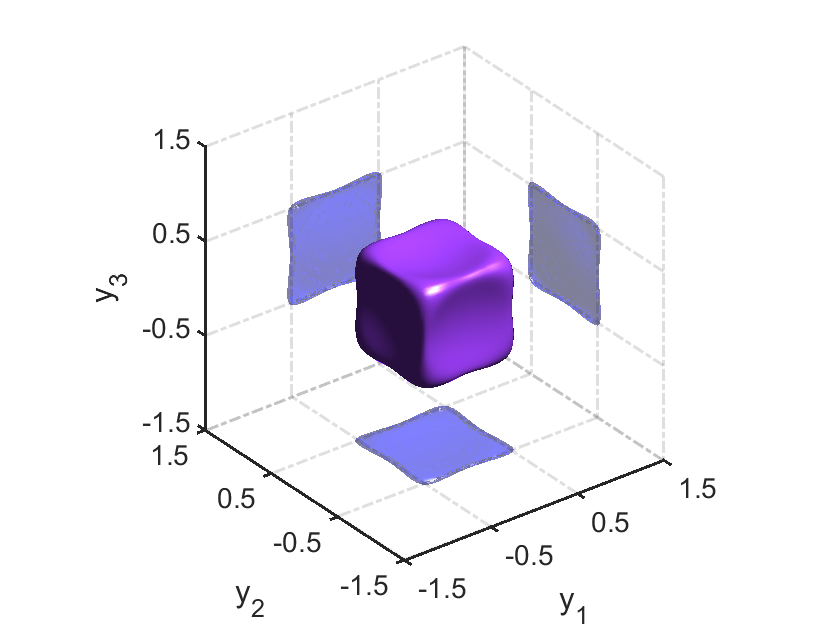}
}
\subfigure[Slices  at $y_2,y_3=0$ ]{
\includegraphics[scale=0.18]{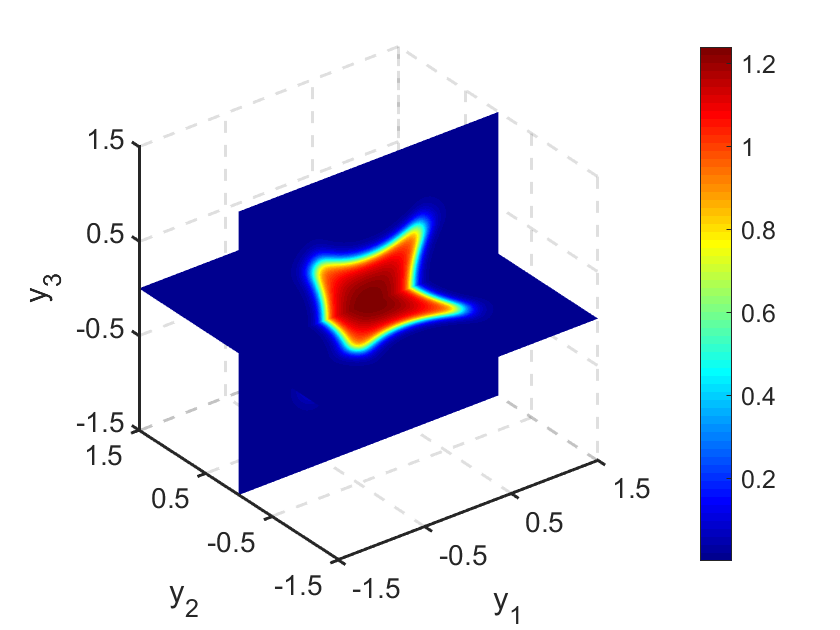}
}
\subfigure[Slices  at $y_1,y_2=0$]{
\includegraphics[scale=0.18]{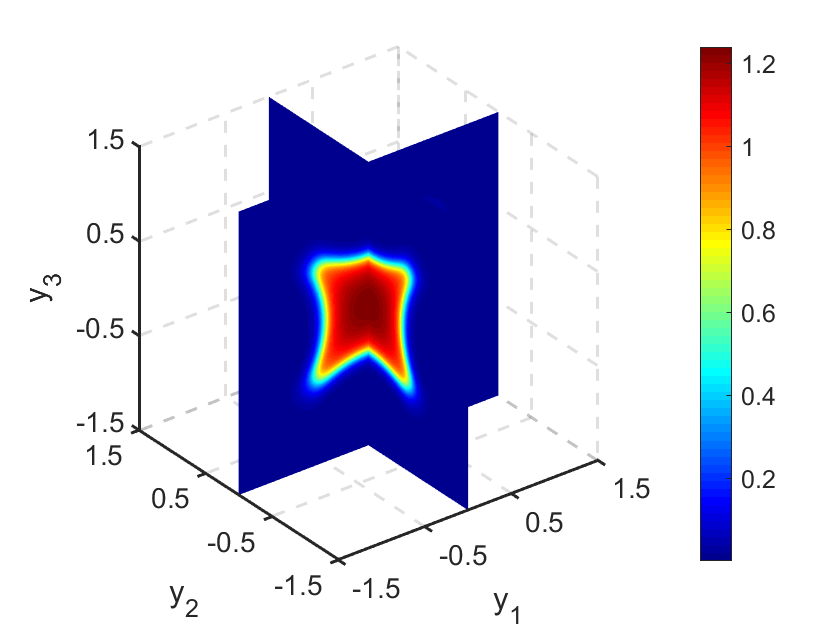}
}
\caption{Reconstructions of a cube from six observation points. %The Fourier transform window is  $(0,3)$.
} \label{fig:near-4}
\end{figure}

\begin{figure}%[H]
\centering
\subfigure[Iso-surface level $=\rot{1.6}$]{
\includegraphics[scale=0.18]{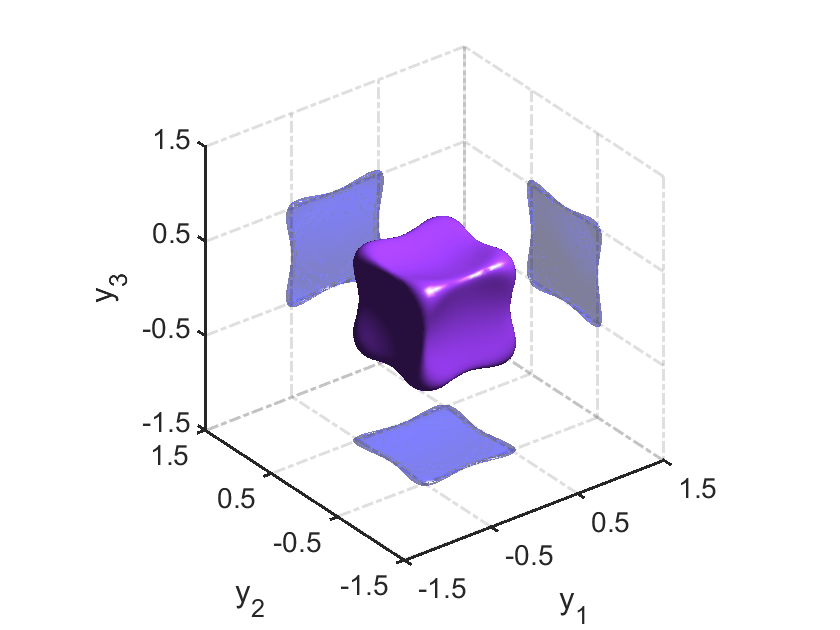}
}
\subfigure[Slices  at $y_2,y_3=0$ ]{
\includegraphics[scale=0.18]{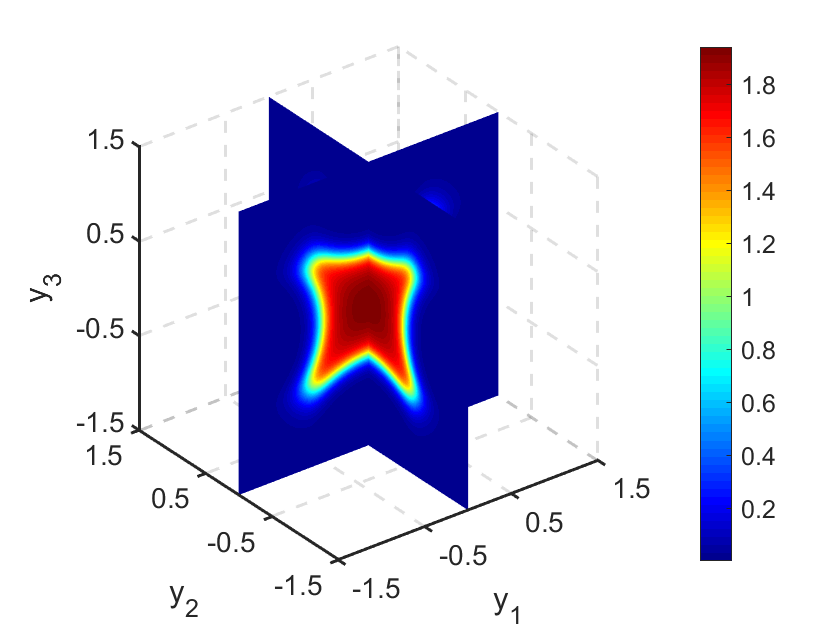}
}
\subfigure[Slices  at $y_1,y_2=0$]{
\includegraphics[scale=0.18]{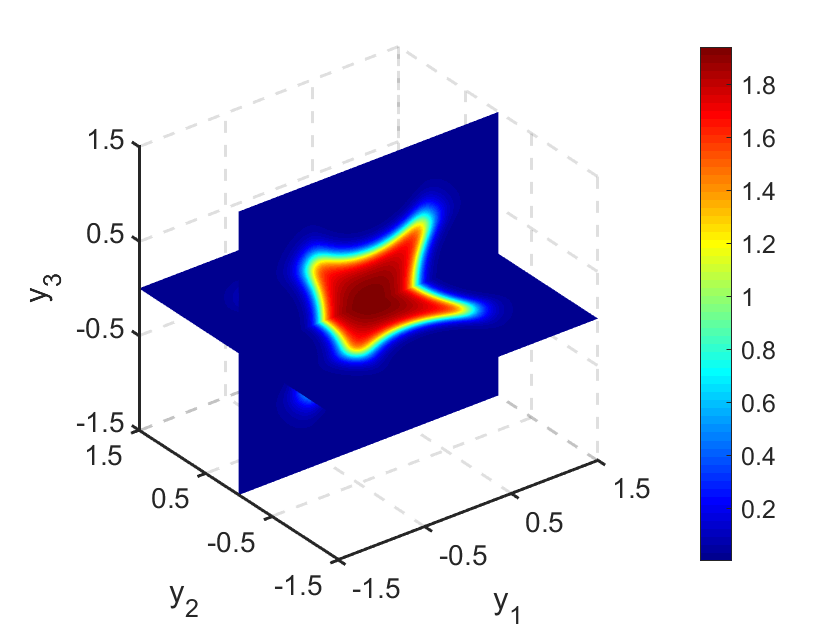}
}
\caption{Reconstructions of a cube with six observation points. %The Fourier transform window is $(0,5)$.
} \label{fig:near-5}
\end{figure}

\subsection {Comparison between far-field and near-field measurements  in $\R^3$}
\rot{We continue to consider the example in Figure \ref{near-3d}, where the exact geometry of the source support is cubic. Theoretically, the geometry can be exactly recovered using three properly-chosen observation directions from the far-field measurements  and approximately recovered using three properly-chosen observation points in $\R^3$. In the numerical examples below, we choose $t_{\min}=0$ and $t_{\max}=0.1$. Using three observation directions $\hat x=\{(1,0,0), (0,1,0),(0,0,1)\}$ in Figure \ref{fig:aa}(a) from the far-field measurements, we can see that both the location and shape of the cubic source support is perfectly reconstructed just as our theoretical results predict.  To clearly illustrate the reconstruction, we also plot the projections of the
images onto the $o y_2 y_3$, $o y_1 y_3$ and $o y_1 y_2$ planes. From the 2D visualizations one sees that the projections are all squares $[-0.5,0.5]^2$. Figure \ref{fig:far-a} shows slices of the reconstruction at the planes $y_1 = 0$, $y_2 = 0$ and $y_3 = 0$.
For comparison we also demonstrate the boundary of the source support
 slice with the pink solid line. These slices also confirm the accuracy of our algorithm. While using three observation points $x=\{(1.5,0,0), (0,1.5,0), (0,0,1.5)\}$ from the near-field measurement in Figure \ref{fig:aa} (b),  we can see that only  the location is captured, but the shape, which is not accurately reconstructed, is even severely disturbed. We also present slices of the reconstruction in Figure \ref{fig:near-a}. The shape of the source support is well-constructed, because the image is formed by the intersection of annulus (just as Figure \ref{fig:near-1} (c)) centered at the three observation points respectively.  However, in the far-field case with one observation point, the source support is located between two planes that are perpendicular to the observation direction.}

\begin{figure}%[H]
\centering
\subfigure[Iso-surface level= $2\times 10^{-4}$]{
\includegraphics[scale=0.2]{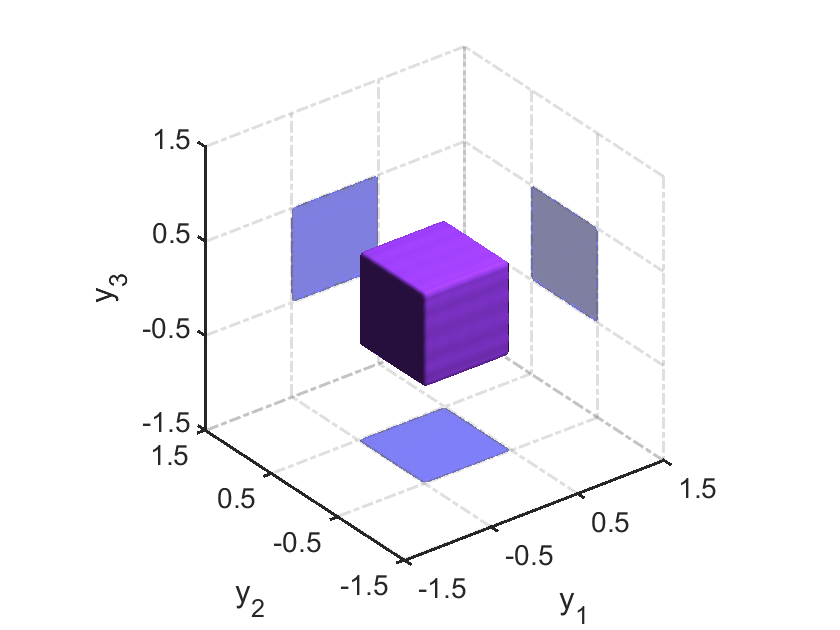}
}  % Iso-surface level= $2\times 10^{-4}$
\subfigure[ Iso-surface level= $1.4\times 10^{-2}$]{
\includegraphics[scale=0.2]{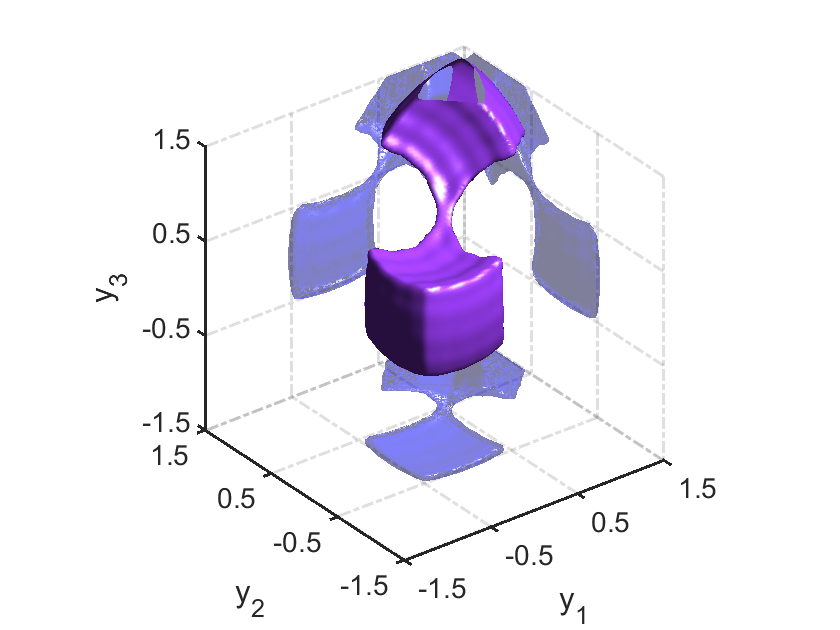}
}  % Iso-surface level= $1.4\times 10^{-2}$
\caption{Reconstructions of a cube from three observation directions with  the far-field measurements in (a) and from three observation points with the near field measurement in (b). %The Fourier transform window is $(0,5)$.
} \label{fig:aa}
\end{figure}

\begin{figure}%[H]
\centering
\subfigure[A slices at $y_1=0$]{
\includegraphics[scale=0.18]{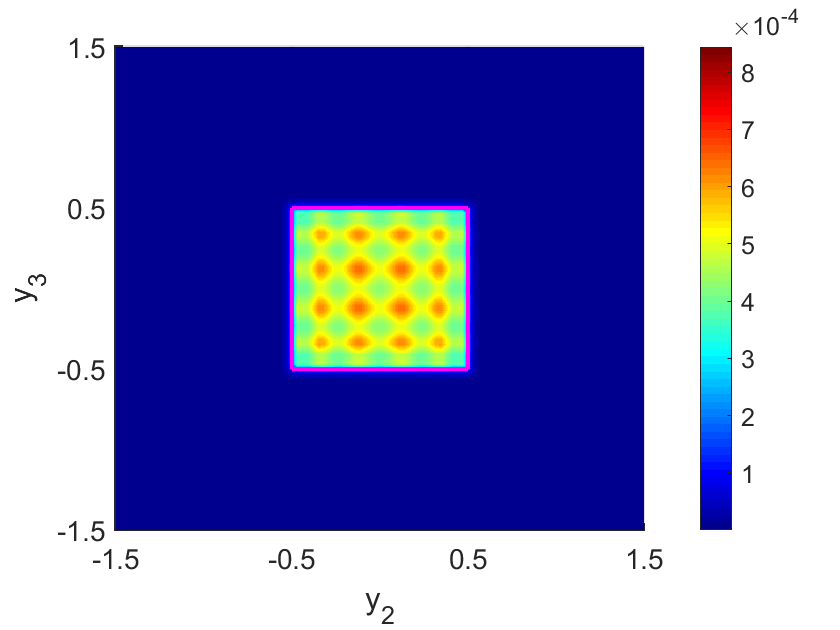}
}
\subfigure[A slices at $y_2=0$ ]{
\includegraphics[scale=0.18]{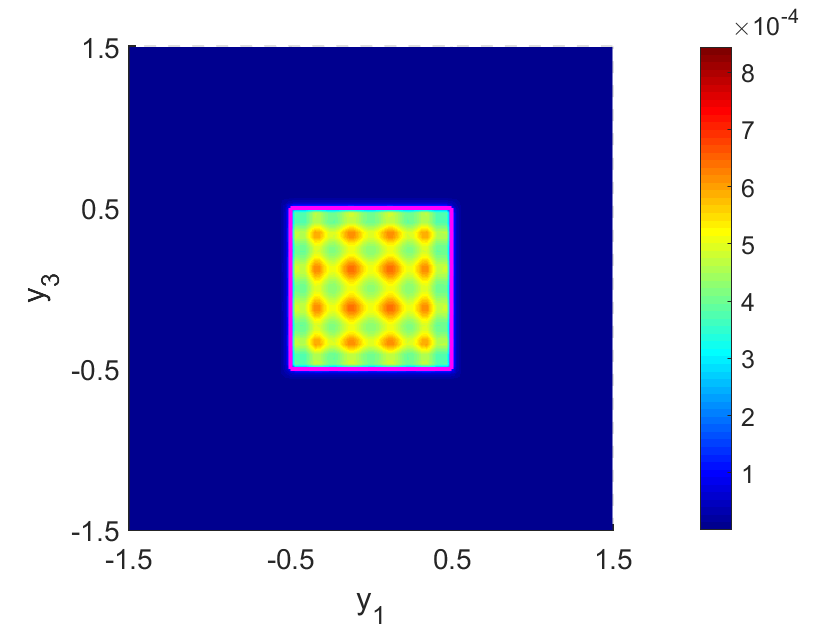}
}
\subfigure[A slices at $y_3=0$]{
\includegraphics[scale=0.18]{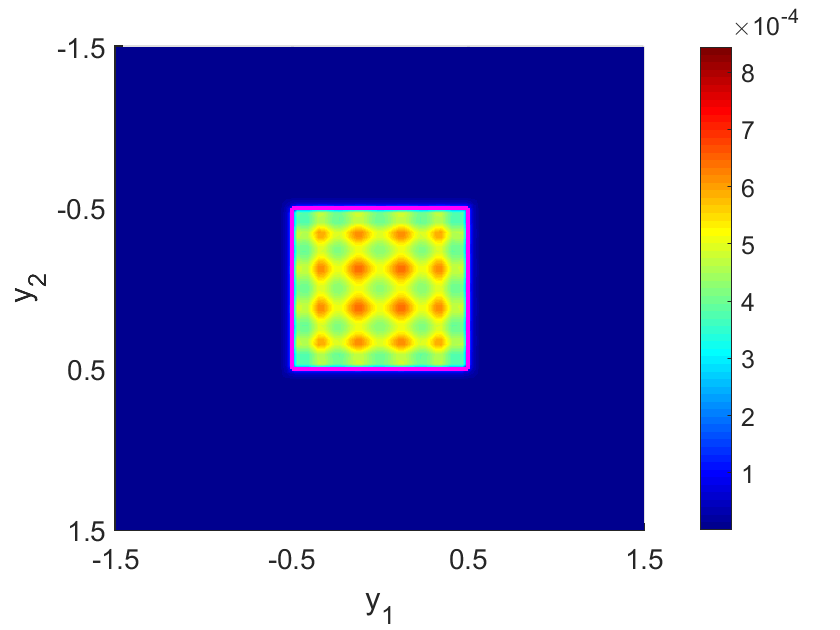}
}
\caption{Reconstructions of a cube with three observation directions for far-field case. %The Fourier transform window is $(0,5)$.
} \label{fig:far-a}
\end{figure}

\begin{figure}%[H]
\centering
\subfigure[A slices at $y_1=0$]{
\includegraphics[scale=0.18]{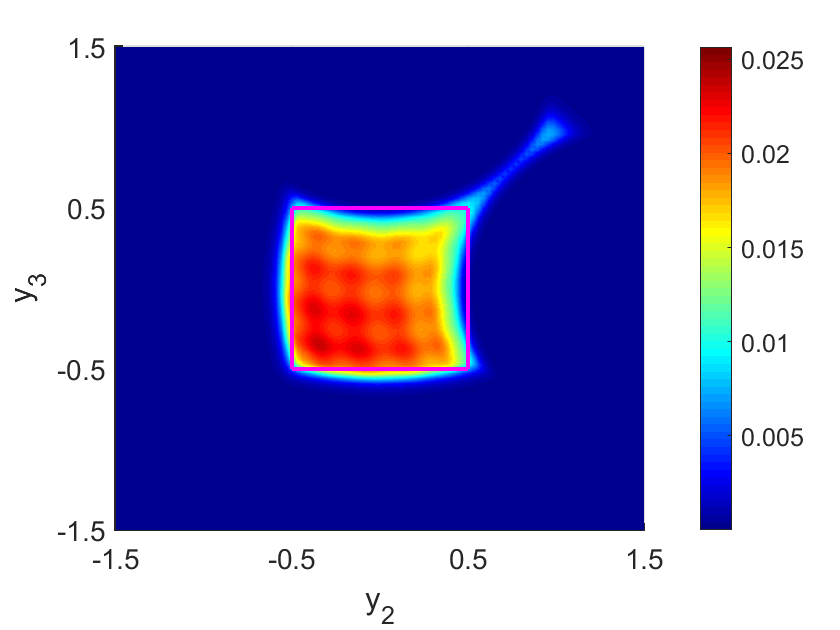}
}
\subfigure[A slices at $y_2=0$ ]{
\includegraphics[scale=0.18]{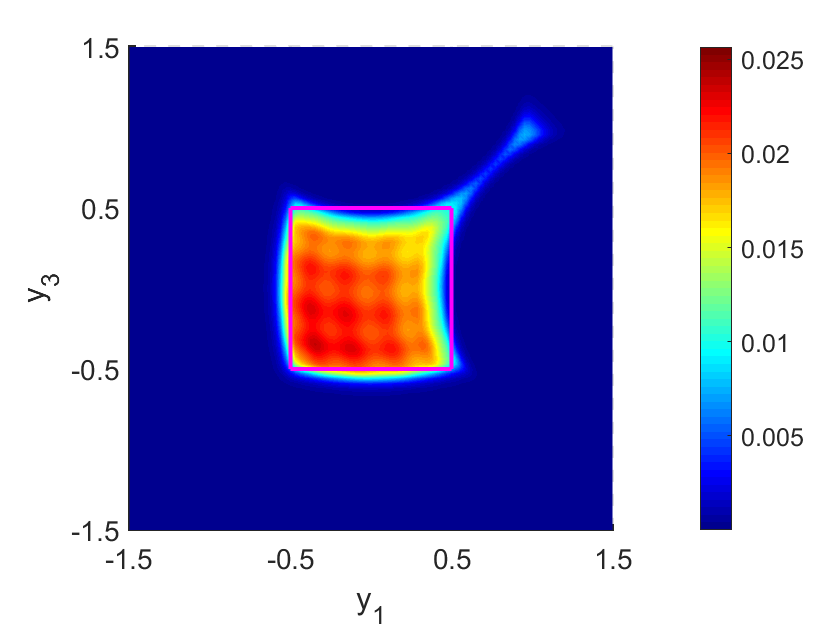}
}
\subfigure[A slices at $y_3=0$]{
\includegraphics[scale=0.18]{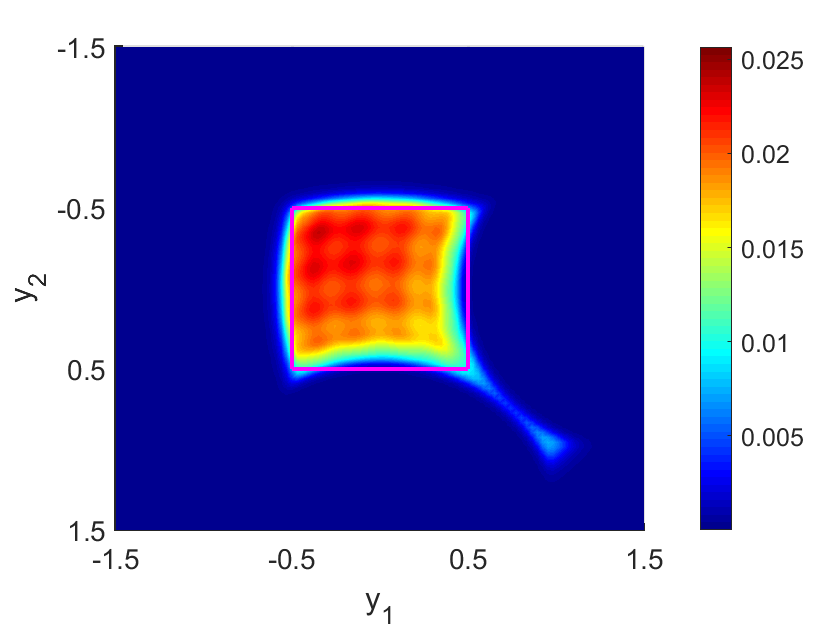}
}
\caption{Reconstructions of a cube with three observation points for near field case. %The Fourier transform window is $(0,5)$.
} \label{fig:near-a}
\end{figure}

\iffalse
\begin{figure}%[H]
\centering
\subfigure[Iso-surface level $=8.5\times 10^{-3}$]{
\includegraphics[scale=0.14]{siamart_220329/Fig-3R/cc-iso-4.png}
}
\subfigure[A slices at $y_1=0$]{
\includegraphics[scale=0.14]{siamart_220329/Fig-3R/cc-1-4.png}
}
\subfigure[A slices at $y_2=0$ ]{
\includegraphics[scale=0.14]{siamart_220329/Fig-3R/cc-2-4.png}
}
\subfigure[A slices at $y_3=0$]{
\includegraphics[scale=0.14]{siamart_220329/Fig-3R/cc-3-4.png}
}
\caption{Reconstructions of a cube with four observation points for near field case. %The Fourier transform window is $(0,5)$.
} \label{fig:far-a-4}
\end{figure}
\fi

\section*{Acknowledgements}
G. Hu acknowledges the hospitality of the Institute for Applied and Numerical
Mathematics, Karlsruhe Institute of Technology and the support of Alexander von
Humboldt-Stiftung. Special thanks are given to R. Griesmaier for stimulating discussions.
The work of G. Hu is partially supported by the National Natural Science Foundation of China (No. 12071236) and the Fundamental Research Funds for Central Universities in China (No. 63213025).

%\end{thebibliography}

\end{document}